\documentclass[final]{svjour3}
\usepackage{amsmath,amsfonts}
\usepackage{graphicx}
\usepackage[shortlabels]{enumitem}
\usepackage[colorlinks=true]{hyperref}
\hypersetup{urlcolor=blue,citecolor=red}
\usepackage[dvipsnames]{xcolor}
\usepackage{tikz}
\usepackage[misc,geometry]{ifsym}

\newtheorem{result}{Result}

\newcommand{\R}{\mathcal{R}_0}

\title{ Population Growth and Competition Models with Decay and Competition
Consistent Delay}
\author{Chiu-Ju Lin${}^1$
    \and Ting-Hao Hsu${}^1$
    \and Gail~S.~K.~Wolkowicz${}^2$
}
\institute{
\Letter\;
Gail~S.~K.~Wolkowicz
\\
\email{wolkowic@mcmaster.ca}
\\[1em]
${}^1$
Department of Mathematics and Statistics, University of New Brunswick, NB, Canada
\\
${}^2$
Department of Mathematics and Statistics, McMaster University, Hamilton, ON, Canada
}
\date{\today}

\begin{document}

\maketitle

%%%%%%%%%%%%%%%%%%%%%%%%%%%%%%
\begin{abstract}
We derive an alternative expression for a delayed logistic equation
in which the rate of change in the population involves a growth rate
that depends on the population density during an earlier time period.
In our formulation,
the delay in the growth term is consistent with
the rate of instantaneous decline in the population given by the model.
Our formulation is a modification of
[Arino et al., J.~Theoret.~Biol.~241(1):109--119, 2006]
by taking the intraspecific competition between the adults and juveniles into account.
We provide a complete global analysis
showing that no sustained oscillations are possible.
A threshold giving the interface between extinction and survival
is determined in terms of the parameters in the model.
The theory of chain transitive sets 
and the comparison theorem for cooperative delay differential equations are used to
	determine the global dynamics of the model.

	We extend our delayed logistic equation
to a system modeling the competition of two species.
For the competition model,
we provide results on local stability, bifurcation diagrams, and adaptive dynamics.
Assuming that the species with shorter delay produces fewer offspring at a time
than the species with longer delay,
we show that there is a critical value, $\tau^*$,
such that the evolutionary trend is for the  delay to approach $\tau^*$.
\end{abstract}

%%%%%%%%%%%%%%%%%%%%%%%%%%%%%%
\section{Introduction}
%%%%%%%%%%%%%%%%%%%%%%%%%%%%%%

The classical logistic equation
was introduced by Verhulst~\cite{Verhulst1838}
as an ordinary differential equation (ODE)
to describe population growth in a limited environment.
Hutchinson~\cite{Hutchinson1948}
noted that the classical logistic equation is not appropriate
when there is a lag in some of the population growth processes,
so he formulated a model as a delay differential equation (DDE) that
 is now known as
the {\em delayed logistic equation}
or {\em Hutchinson's equation}.
However,
Hutchinson's model has been criticized
by ecological modelers
(e.g., Geritz and Kisdi~\cite{Geritz2012}, Nisbet and Gurney~\cite{Nisbet1982}, Arino et al.~\cite{Arino2006})
because the derivation was
not based on clearly defined birth and death process and some of its predictions are
unrealistic, e.g., no matter how long the
delay, the population avoids extinction, and the fina size of the population is
independent of the length of the delay.

Hutchinson's equation is \begin{equation}\label{deq_Hutchinson}
  x'(t)=rx(t)\left(1-\frac{x(t-\tau)}{K}\right),
\end{equation}
where $x(t)$ represents the population density at time $t$,
$r$ is the intrinsic growth rate,
$K$ as the carrying capacity,
and the time lag $\tau$ is a positive constant.
Some authors
(Cooke et al.~\cite{Cooke1999}
and Hadeler and Bocharov \cite{Bocharov2000,Hadeler2003})
have argued that a delay should enter the birth term
rather than the death term and that
the model with delay should take the form
\begin{equation}\label{deq_birth}
  x'(t)=b(x(t-\tau)) x(t-\tau)e^{-\mu_0 \tau}-\mu(x(t))x(t),
\end{equation}
where $b(x)$, $\mu_0$ and $\mu(x)$
are respectively the birth function,
the juvenile death rate, 
and the adult death function.
An equation of the basic form \eqref{deq_birth} was also derived
from age-structured models by 
Gourley and Liu~\cite{Gourley2015} 
and Liu et al.~\cite{Liu2015}.  

For a model
where the population is divided
into two subpopulations, motile and proliferative,
Baker and Rost~\cite{Baker2020}
proposed the equation \begin{equation}\label{deq:Baker2020}
  x'(t)
  =-r x(t)+ rx(t-\tau)\left(
    2-r\int_{t-\tau}^tx(s)\;ds-x(t)
  \right),
\end{equation}
where the integral is related to the population of proliferative agents.

In Arino et al.~\cite{Arino2006}, an approach in the spirit of  Cooke et al.~\cite{Cooke1999} 
and Hadeler and Bocharov \cite{Bocharov2000,Hadeler2003},
proposed a delayed logistic growth model 
\begin{equation}\label{deq:A2006}
  x'(t)= \gamma\frac{ \mu e^{-\mu \tau}
  x(t-\tau)}{\mu  +\kappa (1- e^{-\mu \tau}) x(t-\tau)}
  -\mu\, x(t)- \kappa\, x^2(t),
\end{equation}
in the basic form of \eqref{deq_birth}.  They include the delay in the birth rate term.
However, they  argue further that the delay in the birth or growth rate term 
should not only involve the natural juvenile death rate, but should be consistent with the overall decline terms given by the model.
In the derivation in \cite{Arino2006}, it is assumed that the growth of the
population at time $t$ is proportional (with proportionality constant $\gamma$) to   the number of individuals alive at time
$t-\tau$  that survive until time $t$  avoiding elimination from the
population due to natural death, crowding, or
intraspecific competition with other individuals that were alive at time $t-\tau$
(the rational factor of $\gamma$ in \eqref{deq:A2006}).  They do this in a manner
that is consistent with the natural death, crowding, and intraspecific competition
given modelled by the equation.
However, they ignore decline due to crowding and competition between
those born during the  time interval $t-\tau$ to  $t$. This assumption is reasonable
when those two age groups are living  in very different environments.
This is true for example, for insects that undergo metamorphosis
such as mosquitoes and butterflies.
However,
this assumption is not suitable for mammals,
since then  juveniles and adults share the same environment.

In the new model of delayed logistic growth derived and analyzed in Section~\ref{sec_single},
the competition between
the adults and juveniles is also  taken into account  in a manner consistent with the other terms in the equation.
Thus, we assume
as in the derivation in \eqref{deq:A2006} in Arino et al.~\cite{Arino2006},
that growth in the population 
is proportional to
those individuals alive at time $t-\tau$ that survive until time $t$. 
However, in our derivation, 
the surviving individuals from time $t-\tau$ to time $t$,
avoid natural death, and crowdng and competition with the entire population alive
during that time interval.
In contrast to the discrete delay model
in \eqref{deq:A2006}
this new model has both
discrete and distributed delay terms as in \eqref{deq:Baker2020}.
We determine the global dynamics of the model. In particular,
 all solutions with positive initial data converge to a globally asymptotically
 stable equilibrium with value that has a magnitude that depends on the delay.
If the delay is too long,
this model predicts that the population dies out.
A threshold giving the interface between extinction and survival
is determined in terms of parameters in the model.
The result is consistent with Arino et al.~\cite{Arino2006} and does not suffer
from the issues raised criticizing Hutchinson's model.

One advantage of this model over  \eqref{deq:A2006}, is that it is
possible   to extend it to cover decline due to  competition between two
different populations consistent with the decline terms in the model.
In Section~\ref{sec_competition},
we propose a delay model for two species competition.
We analyze the local stability of the equilibria  and provide bifurcation diagrams for
that model. We also consider adaptive dynamics  using the delay as the evolving trait.

Conclusions and a discussion are given in  Section~\ref{sec_conclusion}.

%%%%%%%%%%%%%%%%%%%%%%%%%%%%%%
\section{The Single Species Model}
\label{sec_single}
%%%%%%%%%%%%%%%%%%%%%%%%%%%%%%

In this section we derive an logistic DDE
by modifying the classical logistic ODE.
The classical logistic ODE can be written as \begin{equation}\label{ode_logistic}
  x'(t)=\gamma x(t)-\mu x(t)- \kappa x(t)^2,
\end{equation}
where $x(t)$ is the population density of the species.
The parameters $\gamma$, $\mu$ and $\kappa$
correspond to growth, death, and intraspecific competition, respectively.

We assume that
the growth rate depends on
the population size some fixed $\tau$ time units in the past.
For each fixed time $t\ge 0$,
we  replace the term $\gamma x(t)$
in \eqref{ode_logistic} by $\gamma X(\tau)$
where $X(\tau)$ denotes the number of individuals alive at time $t-\tau$
that survive until time $t$
(see Fig.~\ref{fig:X}).

In Arino et al.~\cite{Arino2006}
the value of $X(\tau)$ is determined by solving 
\begin{equation}\label{deq:growth_delay_2006}
  X'(s)=-\mu X(s)-\kappa X(s)^2
\end{equation}
for $0\le s\le \tau$, with initial condition  $X(0)=x(t-\tau)$.
The intraspecific competition term $\kappa X(s)^2$
is based on the assumption that
the population $X(s)$ only competes
with $X(s)$ themselves rather than the entire population $x(t-\tau+s)$.
This assumption is reasonable for metamorphosis insects,
for which the juveniles and adults have very different living environment
and those two age groups cannot compete with each other.
However,
this assumption is not suitable for many animals such as mammals,
for which juveniles and adults share the same environment.
Therefore, we assume that the population $X(s)$
competes with the entire population $x(t-\tau+s)$.
Hence, we modify \eqref{deq:growth_delay_2006}
 as follows: 
\begin{equation}\label{deq:growth_delay}
  X'(s)=-\mu X(s)-\kappa X(s)\,x(t-\tau+s)
\end{equation}
for $0\le s\le \tau$.
Equation \eqref{deq:growth_delay} with initial condition $X(0)=x(t-\tau)$
can be solved explicitly
and has solution 
\begin{equation}\label{Xtau}
  X(\tau)
  =x(t-\tau)\exp\left(-\mu \tau-\kappa\int_{t-\tau}^{t}x(v)dv\right).
\end{equation}
Replacing the term $\gamma x(t)$ in \eqref{ode_logistic}
by $\gamma X(\tau)$,
where $X(\tau)$ is given in \eqref{Xtau},
we arrive at
\begin{equation}\label{deq_single}
  x'(t)
  =\gamma x(t-\tau)\exp\left(-\mu \tau-\kappa\int_{t-\tau}^{t}x(v)dv\right)
  -\mu\, x(t) -\kappa\, x^2(t).
\end{equation}
with initial data, $x(t)=\phi(t)\in C([-\tau,0],\mathbb{R}_+\setminus\{0\})$.

We call equation~\eqref{deq_single}
the {\em mixed alternative logistic DDE}.
We will discuss basic properties of \eqref{deq_single}
in Section~\ref{sec_single_local}
and provide global dynamics in Section~\ref{sec_single_global}.

\begin{figure}[tb!]
\centering
\includegraphics[trim = 0cm 2cm 0cm 0cm, clip, width=.6\textwidth]{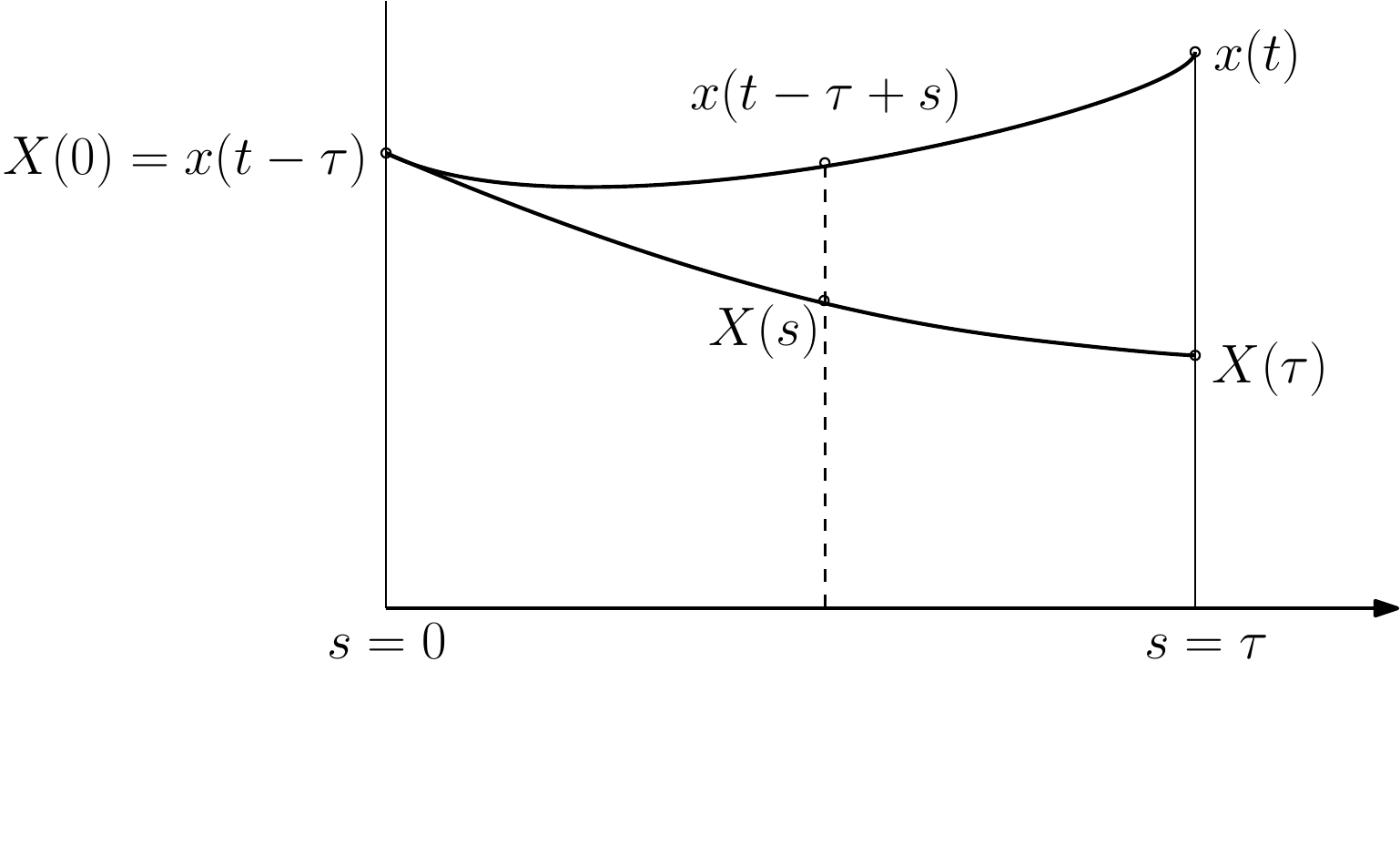}
\caption{%
The relationship between the populations $x$ and $X$ for $s\in[0,\tau]$.
Shown is  $x(t-\tau+s)$, the total population 
alive at time $t-\tau+s$ where $s\in[0,\tau]$  
and  $X(s)$, the population of
individuals alive at time $t-\tau+s$ that were alive at time $t-\tau$
and survived until time $t-\tau+s$. 
The two populations are related by \eqref{deq:growth_delay}
with initial condition $X(0)=x(t-\tau)$ that has explicit solution given by \eqref{Xtau}.
}\label{fig:X}
\end{figure}

\begin{remark}
In Arino et al.~\cite{Arino2006},
the individuals that survive from time $t-\tau$ to time $t$ is given by the
	solution of  \eqref{deq:growth_delay_2006},
and the resulting logistic model is given by equation \eqref{deq:A2006}.
The dynamics of \eqref{deq:A2006} is determined by a survival threshold
\begin{equation}\label{eq:tau_h}
	\tau_H=\frac{1}{\mu}\ln\left(\frac{\gamma}{\mu}\right).
\end{equation}
If the time delay is too large, i.e., $\tau\geq \tau_H$, the population dies out, as one might expect.
If the time delay is small enough, i.e.,  $\tau<\tau_H$,
then the populations size  approaches  a positive equilibrium, interpreted as the
	delay reduced carrying capacity, since it is a decreasing
	function of the delay. 
\end{remark}

%%%%%%%%%%%%%%%%%%%%%%%%%%%%%%
\subsection{Equilibria and Local Stability}
\label{sec_single_local}
%%%%%%%%%%%%%%%%%%%%%%%%%%%%%%

Following the proof of Arino et.~al~\cite[Proposition 3.1]{Arino2006},
together with the observation that solutions of \eqref{deq_single} satisfy
\begin{equation}\label{eq:well-posed}
  x'(t) \le \gamma e^{-\mu\tau} x(t-\tau_i)-\mu x(t)-\kappa x^2(t),
\end{equation}
it can be shown that for
any solution  of \eqref{eq:well-posed}
with initial data $\phi(t)\in C([0,\tau], \mathbb R_+\setminus\{0\})$,
a constant $M$  exists such that \[
  M> \max_{t\in [0,\tau]}\phi(t)
  \quad\text{and}\quad
  \gamma e^{-\mu\tau}-\mu - \kappa M <0,
\]
and the solution $x(t)$
remains positive and bounded above by $M$ for all $t>0$.
By the standard theory of delay differential equations
(see e.g.\ Hale and Verduyn~\cite{Hale1993}),
it follows that model~\eqref{deq_single} is well-posed,
i.e.,
every solution
with  positive initial data remains positive and
is eventually bounded above
by $K=(\gamma e^{-\mu\tau}-\mu)/\kappa$, a decreasing function of the delay,
$\tau$. As in Arino et al.~\cite{Arino2006}, we  also interpret this value as the delay reduced carrying capacity.

The extinction equilibrium $x=0$ always exists.
The nontrivial equilibrium of \eqref{deq_single} is determined by
\begin{equation}\label{def_xbar}\begin{aligned}
  \gamma e^{-\tau(\mu+\kappa{x})}-(\mu+\kappa{x})=0.
\end{aligned}\end{equation}
Since the left-hand side of \eqref{def_xbar}
decreases unboundedly as $x$ increases
and equals $\gamma e^{-\tau \mu}-\mu$ when $x=0$,
equation \eqref{def_xbar} has a unique positive root
if and only if $\gamma e^{-\mu\tau}>\mu$.
This is equivalent to the condition, $\tau<\tau_H$, with $\tau_H$ defined in \eqref{eq:tau_h}.

We define
\begin{equation}\label{def_R0}
  \mathcal R_0=\frac{\gamma e^{-\mu\tau}}{\mu}.
\end{equation}
Note that $\R{}>1$ if and only if $\tau<\tau_H$.
Equation \eqref{deq_single}
always has the trivial equilibrium $0$.
A positive equilibrium $x^*$
exists if and only if $\R>1$,
where $\R$ is given in \eqref{def_R0}.
In this section we show that
whenever the positive equilibrium exists,
it is locally asymptotically stable.

\begin{lemma}
Consider equation~\eqref{deq_single}.
Let $\R{}$ be defined by \eqref{def_R0}.
\begin{itemize}
\item[(a)]
If $\R{}<1$,
then the trivial equilibrium $x=0$ is
locally asymptotically stable
and there is no positive equilibrium.
\item[(b)]
If $\R{}>1$,
then the trivial equilibrium $x=0$ is unstable
and there exists a unique positive equilibrium $x^*>0$.
Moreover, $x^*$ is locally asymptotically stable.
\end{itemize}
\label{lem_local_single}
\end{lemma}

\begin{proof}
Throughout this proof, we scale $x\mapsto \kappa x$
in equation~\eqref{deq_single}, and hence without loss of generality, assume that $\kappa=1$.

First we show that
the local stability of an equilibrium $\bar{x}$ of \eqref{deq_single}
can be determined by its characteristic equation
 \begin{equation}\label{char_xbar}\begin{aligned}
  \gamma e^{-\tau(\mu+{\bar{x}})}\left(
  e^{-\lambda\tau}+\frac{e^{-\lambda\tau}-1}{\lambda}\,{ \bar{x}}
  \right)
  -\mu-2{\bar{x}}
  -\lambda
  =0
\end{aligned}\end{equation}

Let $\mathcal{C}$ be the set of all bounded continuous functions on $[-\tau,\infty)$.
For any function $x\in \mathcal{C}$ and $t\ge 0$,
we denote $x_t$ the function in $\mathcal{C}$ defined by
\begin{equation}\notag\begin{aligned}
  x_t(s)= x(s+t),\quad s\in [-\tau,\infty).
\end{aligned}\end{equation}
Let $F(x_t)$ be the right-hand side of \eqref{deq_single} for any $x\in \mathcal{C}$.
Then $F$ maps $\mathcal{C}$ into $\mathcal{C}$.
If $\bar{x}\in \mathcal{C}$ satisfies $F(\bar{x})=0$,
then the linearization $DF(\bar{x})$ of $F$ at $\bar{x}$
is given by \begin{equation}\label{DF_xbar}\begin{aligned}
  &\big[DF(\bar{x})(\varphi)\big](t)
  \\
  &=\gamma e^{-\tau\mu-\int_{t-\tau}^t\bar{x}(s)\;ds}\left[
    -\left(\int_{t-\tau}^t\varphi(s)\;ds\right)\bar{x}(t-\tau)
    + \varphi(t-\tau)
  \right]
  -\big(\mu+2\bar{x}(t)\big)\varphi(t)
\end{aligned}\end{equation}
for all $\varphi\in \mathcal{C}$.
Taking $\bar{x}$ to be the constant function $\bar{x}$,
the eigenvalue problem
$\big[DF(\bar{x})(\varphi)\big](t)=\lambda\varphi(t)$
yields \begin{equation}\label{DF_xhat}
  \lambda \varphi(t)
  =\gamma e^{-\mu\tau-\bar{x}\tau}\left[
  -\bar{x} \left(\int_{t-\tau}^t\varphi(s)\;ds\right)
  + \varphi(t-\tau)
  \right]
  -\big(\mu+2\bar{x}\big)\varphi(t).
\end{equation}
It can be shown
by applying the Laplace transform to \eqref{DF_xhat}
(see e.g.\ \cite[Section 1.5]{Hale1993})
that $\varphi=e^{\lambda t}$
and that
\[
  \lambda e^{\lambda t}
  =\gamma e^{-\mu\tau-\bar{x}\tau}\left[
  -\bar{x} \left( \frac{1-e^{-\tau}}{\lambda}e^{\lambda t}\right)
  + e^{\lambda(t-\tau)}
  \right]
  -\big(\mu+2\bar{x}\big)e^{\lambda t},
\]  leading to \[
  \left[
    \gamma e^{-\tau(\mu+{\bar{x}})}\left(
    e^{-\lambda\tau}+\frac{e^{-\lambda\tau}-1}{\lambda}\,{x_0}
    \right)
    -\mu-2{\bar{x}}
    -\lambda
  \right] e^{\lambda t}=0.
\] Hence all eigenvalues of $DF(\bar{x})$ are roots of \eqref{char_xbar}.
Therefore,
if all roots of  equation \eqref{char_xbar}
have negative real parts,
then $\bar{x}$ is locally asymptotically stable
and if any root has a positive  real part positive, then $\bar{x}$ is unstable.

At $\bar{x}=0$,
the linearization \eqref{DF_xbar} is \begin{equation}\label{dphi_x0}
  \varphi'(t) = \gamma e^{-\tau\mu}\varphi(t) - \mu\varphi(t)
\end{equation}
and the characteristic equation \eqref{char_xbar} at $\bar{x}=0$ is
\begin{equation}\label{char_x0}
  \gamma e^{-\tau\mu} e^{-\lambda\tau} -\mu  -\lambda  =0.
\end{equation}
Note that \eqref{dphi_x0} coincides with
the linearization of \eqref{deq:A2006} at $x=0$
(see \cite[Appendix B]{Arino2006}).
Applying Hayes Theorem \cite{Hayes1950},
it follows that all roots of  equation \eqref{char_x0}
have negative real parts if $\gamma e^{-\mu\tau}<\mu$,
and some root of the equation
has a positive real part positive if $\gamma e^{-\mu\tau}>\mu$.
This proves (a) and the first part of (b).

To prove the second part of (b), next we assume that $\R{}>1$.
Then a positive equilibrium, $x^*>0$, for equation \eqref{deq_single} exists.
Using equation \eqref{def_xbar},
the characteristic equation \eqref{char_xbar} at $\bar{x}=x^*$
can be written as
\begin{equation}\notag\begin{aligned}
  (\mu+{x^*})\left(
  e^{-\lambda\tau}+\frac{e^{-\lambda\tau}-1}{\lambda}\,{x^*}
  \right)
  -\mu-2 {x^*}
  -\lambda=0.
\end{aligned}\end{equation}
Note that $\lambda=0$ is not a root
since the left-hand side of the equation
has the nonzero limit, $-[(\mu+x^*)\tau+1]x^*$, as $\lambda\to 0$.
Writing $m=\mu+ x^*$, the above equation becomes \[
  m(e^{-\lambda\tau}-1)
  -\lambda
  + x^*\left(
    m\frac{e^{-\lambda\tau}-1}{\lambda}-1
  \right)
  =0,
\] or
\begin{equation}\label{lambda_m}
  \left(m\,\frac{e^{-\lambda\tau}-1}{\lambda}-1\right)(\lambda+{x^*})=0.
\end{equation}
The second factor
of the left-hand side of \eqref{lambda_m}
has a single root $\lambda=- x^*<0$.
Let $\lambda=\alpha+i \beta$ be a root of the first factor
of the left-hand side of \eqref{lambda_m}.
Then \[
  m\Big[
    (e^{-\alpha\tau}\cos \beta\tau-1)
    -(e^{-\alpha\tau}\sin \beta\tau) i
  \Big]
  =\alpha+i\beta.
\]
Hence $\alpha<0$.
We conclude that whenever $\mathcal{R}_0>1$, $x^*>0$ exists and is local asymptotically stable.
\qed
\end{proof}

%%%%%%%%%%%%%%%%%%%%%%%%%%%%%%
\subsection{Global Dynamics}
\label{sec_single_global}
%%%%%%%%%%%%%%%%%%%%%%%%%%%%%%

In Section~\ref{sec_single_local} we see that
the positive equilibrium $x^*$ of \eqref{deq_single}
is locally asymptotically stable whenever it exists.
In this section we show that
in this case $x^*$ attracts all positive solutions
and is thus globally asymptotically stable.

The main difficulty in the proof
arises from the fact that equation~\eqref{deq_single}
does not have a positive delayed feedback,
that is,
the right hand side of the equation
is not an increasing function of the delayed term.
In contrast, the alternative logistic DDE \eqref{deq:A2006}
has a positive delayed feedback,
so the proof in Arino et al.~\cite{Arino2006} of the global stability
of the positive equilibrium of \eqref{deq:A2006}
is not applicable.

\begin{theorem}
Consider equation~\eqref{deq_single}.
Let $\R{}$ be defined by \eqref{def_R0}.
\begin{itemize}
\item[(a)]
If $\R{}\leq1$, then $0$
attracts all nonnegative solutions.
Moreover, if $\R{}<1$,
then $0$  is globally asymptotically stable.
\item[(b)]
If $\R{}>1$, then the unique positive equilibrium $x^*$
is globally asymptotically stable.
\end{itemize}
\label{thm_single}
\end{theorem}

\begin{proof} {\it (a)}
First we assume that $\R\le 1$.
Let $x(t)$ be any nonnegative solution of \eqref{deq_single}.
We claim that $x(t)$ converges as $t\to\infty$.
Suppose that $x(t)$ does not converge.
Then
\[
  \hat{x}:=\limsup_{t\to\infty}x(t)>\liminf_{t\to\infty}x(t)\geq0.
\]
By the fluctuation lemma \cite{Hirsch1985},
there exists an increasing sequence of times $\{t_n\}$ such that
$\lim_{n\to\infty}x(t_n)=\hat{x}$ and $x'(t_n)=0$ for all $n$.
By \eqref{deq_single} it follows that \begin{equation}\notag
  0
  =\gamma \exp\left(-\mu\tau-\kappa\int_{t_n-\tau}^{t_n}x(s)\;ds\right)x(t_n-\tau)
  -\mu \hat{x}-\kappa \hat{x}^2.
\end{equation}
Hence \begin{equation}\notag
  0
  <\gamma e^{-\mu \tau}x(t_n-\tau)
  -\mu \hat{x}-\kappa \hat{x}^2.
\end{equation}
Taking limit superior as $n\to \infty$, we obtain 
\begin{equation}\notag
\begin{aligned}
  0
  &\le \gamma e^{-\mu \tau}\hat{x}
  -\mu \hat{x}-\kappa \hat{x}^2
  \\
  &= \hat{x}\big(
    \mu(\R{}-1)-\kappa\hat{x}
  \big)<0,
\end{aligned}\end{equation}
 a contradiction.
Hence $\lim_{t\to \infty}x(t)$ exists.
By Lemma~\ref{lem_local_single},
 system \eqref{deq_single}
has no positive equilibrium when $\R\le 1$.
Hence, we conclude that all nonnegative solutions converge to $0$.

\noindent 
	{\it (b)} Next, we assume that $\R> 1$.
For any nonzero nonnegative solution $x(t)$ of \eqref{deq_single},
we introduce the Lyapunov-type function \begin{equation}\notag%\label{def_Y}
  Y(t)=\gamma
  \exp\left(
    -\mu\tau-\kappa \int_{t-\tau}^t x(s)ds
  \right)
  -\mu- \kappa x(t).
\end{equation}
Then \eqref{deq_single} yields
\begin{equation}\notag\begin{aligned}
  Y'(t)
  &=
  \gamma \kappa
  \exp\left(
    -\mu\tau-\kappa \int_{t-\tau}^t x(s)ds
  \right)
  \Big[- x(t)+ x(t-\tau)\Big]
  \\
  &\qquad
  -\kappa \left[
    \gamma
    \exp\left(
      -\mu\tau-\kappa \int_{t-\tau}^t x(s)ds
    \right)
     x(t-\tau)
    -\mu x(t)- \kappa x^2(t)
  \right]
  \\
  &=- \kappa\, x(t)Y(t).
\end{aligned}\end{equation}
Hence $Y(t)$ converges  monotonically  to $0$ as $t$ tends to infinity,
that is
\begin{equation}\label{eq:Y}
  \lim_{t\to\infty}\Big[
    \gamma
    \exp\left(
      -\mu\tau-\kappa \int_{t-\tau}^t x(s)ds
    \right)
    -\mu- \kappa\, x(t)
  \Big]=0.
\end{equation}

By the theory of chain transitive sets (see e.g.~\cite{Hirsch2001}),
to study the global dynamics of \eqref{deq_single},
from \eqref{eq:Y} it is enough to study
the restriction of \eqref{deq_single}
on the set of functions that
satisfy the integral equation
\begin{equation}\label{integral_single}
  \kappa\, x(t)= \gamma
  \exp\left(
      -\mu\tau
      -\kappa\int_{t-\tau}^t x(s)ds
    \right)
    -\mu.
\end{equation}
The restriction of \eqref{deq_single}
on the set of functions satisfying \eqref{integral_single}
can be written as \[
  x'(t)
  =x(t-\tau)\big[\mu+\kappa\,x(t)\big]
  -\mu x(t) -\kappa x^2(t),
\] or, equivalently, \begin{equation}\label{deq_single_limit}
  x'(t)
  =\big[\mu+\kappa\, x(t)\big]
  \,\big[x(t-\tau)-x(t)\big].
\end{equation}

For any positive solution $x(t)$ of \eqref{deq_single_limit} that satisfies \eqref{integral_single},
we  show that $x(t)$ converges to $x^*$.
Since the only constant solution satisfying \eqref{integral_single} is $x^*$,
it suffices to show that $x(t)$ converges as $t\to \infty$.

We set functions \[
  \varphi_n(s)=x(n\tau+s),\;\;
  s\in [0,\tau],
\]
and numbers \[
  A_n=\max_{s\in[0,\tau]} \varphi_n(s),
\]
for $n=0,1,2,\dots$.
Note that
equation \eqref{deq_single_limit}
has a positive delayed feedback.
Since each constant function $A_n$ is a solution of \eqref{deq_single_limit},
by the comparison theorem
for cooperative delay differential equations
(\cite[Theorem 5.1.1]{Smith1995}), \[
  \varphi_{n+1}(s)\le A_n,\;\;
  s\in [0,\tau],
\]
for $n=0,1,2,\dots$.
It follows that $\{A_n\}$ is a decreasing sequence.
Hence,
\begin{equation}\label{limit_An}
  \lim_{n\to \infty}A_n
  = A,
\end{equation}
where $A=\limsup_{t\to \infty} x(t)$.
To show that $\lim_{t\to \infty} x(t)$ exists,
we  show that $\{\varphi_n\}_n$ converges to a constant
(see Fig.~\ref{fig:limiting}(a)).
We claim that every subsequence of $\{\varphi_n\}_n$
has a further subsequence that converges to
the constant function $A$
in the $C([0,\tau])$-norm.

\begin{figure}[hbtp]\centering
  (a)\hspace{17em}(b)
  \\
  \includegraphics[width=0.48\textwidth]{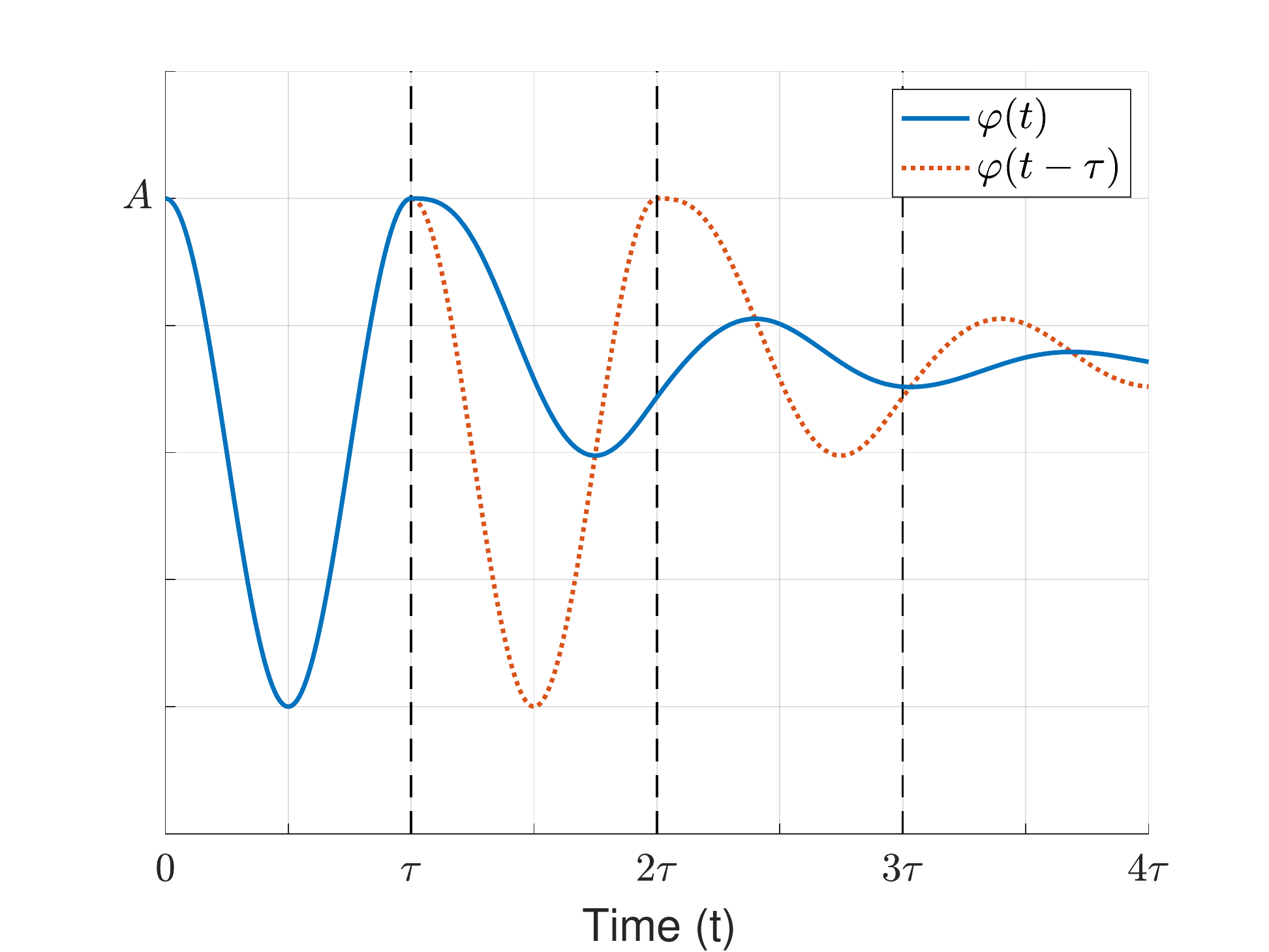}
  \includegraphics[width=0.48\textwidth]{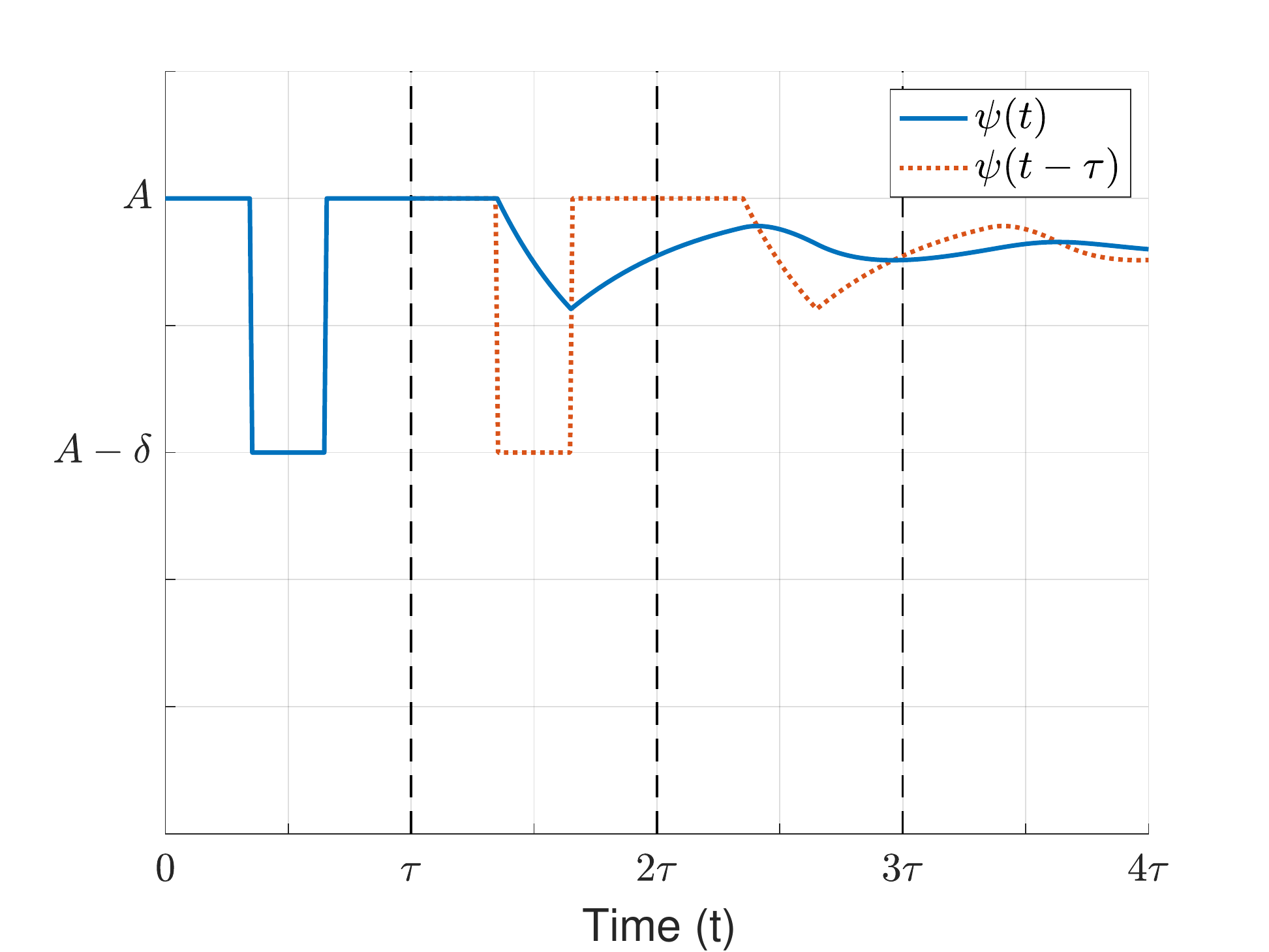}
\caption{(a)
The function $\varphi_n(s)$ is
the restriction of $x(t)$ on the interval $[n\tau,(n+1)\tau]$.
The maximum value $A_n$ of $\varphi_n(t)$
is a decreasing sequence.
(b) The function $\psi(t)$ is defined by \eqref{def_psi} on $[0,\tau]$
and by \eqref{deq_single_limit} for $t\ge \tau$.
When $t\ge \tau$,
$\psi'(t)$ has the same sign as $\psi(t-\tau)-\psi(t)$.
The maximum value of $\psi$ on $[2\tau,3\tau]$
is strictly less than that on $[0,\tau]$.
}
\label{fig:limiting}
\end{figure}

Given any subsequence of $\{\varphi_n\}$,
since the sequence of functions $\{\varphi_n\}$
is uniformly bounded on the compact set $[0,\tau]$,
by the Arzela-Ascoli theorem
there is a uniformly convergent
further subsequence $\{\varphi_{n_j}\}$.
Let $\varphi$ be the limit of $\{\varphi_{n_j}\}$.
We extend the domain of $\varphi(s)$
by setting $\varphi(t)$ to be a solution of \eqref{deq_single_limit} for $t\ge \tau$.
Note that for each $k=0,1,2,\dots$,
the sequence $\{\varphi_{n_j+k}(s)\}_j$
converges uniformly to $\varphi(k\tau+s)$.
By \eqref{limit_An},
it follows that
\begin{equation}\label{max_phi_k}
  \max_{s\in [0,\tau]}\varphi(k\tau+s)= A
\end{equation} for $k=0,1,2,\dots$.
To prove that $\varphi$ is a constant function,
we  show that \eqref{max_phi_k}
fails to hold for some $k$ if $\varphi$ is non-constant.

Suppose by contradiction that $\varphi$ is non-constant.
Then there exists $0<s_1<s_2<\tau$ and $\delta>0$ such that \[
  \varphi(s)<A-\delta,
  \quad\text{for}\;\;s\in(s_1,s_2).
\]
We define
(see Fig.~\ref{fig:limiting}(b))
\begin{equation}\label{def_psi}
  \psi(s)=\begin{cases}
    A,& \text{for}\;\;s\in [0,s_1]\cup [s_2,\tau],
    \\
    A-\delta,& \text{for}\;\;s\in (s_1,s_2).
  \end{cases}
\end{equation}
Then $\varphi(s)\le \psi(s)$ for $s\in [0,\tau]$.
We extend the domain of $\psi(s)$ for $t\ge \tau$
by setting $\psi(t)$ to be a solution of \eqref{deq_single_limit} for $t\ge \tau$.
With $\psi(\tau+t)$ playing the role of $x(t)$ in equation~\eqref{deq_single_limit},
the values of $\psi'(\tau+s)$ and $\psi(s)-\psi(\tau+s)$ have the same sign.
By standard ODE theory,
it is straightforward to show that
\begin{equation}\label{psi_max}
  \psi(\tau+s)\begin{cases}
    =A,& \text{for}\;s\in [0,s_1],
    \\
    <A,&\text{for}\;s\in (s_1,\tau].
  \end{cases}
\end{equation}
With $\psi(2\tau+t)$ playing the role of $x(t)$ in  equation~\eqref{deq_single_limit},
the values of $\psi'(2\tau+s)$ and $\psi(s)-\psi(2\tau+s)$ have the same sign.
By \eqref{psi_max} and the standard ODE theory,
it is straightforward to show that
\begin{equation}\label{psi2_max}
  \psi(2\tau+s)< A
  \quad\text{for}\;\;s\in [0,\tau].
\end{equation}

Since $\varphi(s)\le \psi(s)$ for $s\in [0,\tau]$,
by the comparison theorem
we have $\varphi(s)\le \psi(s)$ for $s\in [0,3\tau]$.
By \eqref{psi2_max} it follows that \begin{equation}\notag
  \max_{s\in [0,\tau]}\varphi(2\tau+s)<A,
\end{equation}
contradicting  \eqref{max_phi_k}.
This implies that $\varphi$ is a constant function.
Hence, $x(t)$ converges to a constant as $t\to \infty$.
We conclude that $\lim_{t\to \infty}x(t)=x^*$.
\qed
\end{proof}

%%%%%%%%%%%%%%%%%%%%%%%%%%%%%%
\section{A Competition Model}
\label{sec_competition}
%%%%%%%%%%%%%%%%%%%%%%%%%%%%%%

In this section,
we derive a system modeling the competition between two species, $x_1$ and $x_2$.
For $i=1,2$, we denote the growth rate of the $i$th species by $\gamma_i$,
death rate by $\mu_i$,
decay-consistent delay by $\tau_i$,
the intra-specific competition parameter by $\kappa_i$,
and the inter-specific competition parameter by $\alpha_i$.

For each fixed time $t\ge 0$,
we  denote by $X_1(s)$, the number of individuals
of species $x_1$ alive at time $t-\tau_1$ that survive until time $t-\tau_1+s$.
Assuming decay consistent until decay rate of species $x_1$
is a combination of the natural death $\mu_1X_1(s)$,
the intra-specific competition $ X_1(s)x_1(t-\tau_1+s)$,
and the inter-specific competition $\alpha_2X_1(s)x_2(t-\tau_1+s)$,
we assume that $X_1(s)$ satisfies the following ODE
\begin{equation}\notag%\label{deq:X1}
  X_1'(s)=-\mu_1X_1(s)-\kappa_1 X_1(s)x_1(t-\tau_1+s)-\alpha_2X_1(s)x_2(t-\tau_1+s)
\end{equation}
with initial condition $X_1(0)=x_1(t-\tau_1)$.
This has explicit solution  \[
  X_1(\tau_1)
  =x_1(t-\tau_1)
  \exp\left(
    -\mu \tau_1
    -\displaystyle\int_{t-\tau_1}^t \Big(\kappa_1x_1(s)+\alpha_2 x_2(s)\Big) ds
  \right).
\]
Similarly,
we set $X_2(s)$ to be the number of individuals of species $x_2$ that survive from
time $t-\tau_2$ to  $t-\tau_2+s$
and assume that
\[
  X_2(\tau_2)
  =x_2(t-\tau_2)
  \exp\left(
    -\mu \tau_2
    -\int_{t-\tau_2}^t \Big(\kappa_2x_2(s)+\alpha_1 x_1(s)\Big) ds
  \right).
\]
Therefore, we propose the two-species competition model:
\begin{equation}\label{deq_competition}\begin{aligned}
  x_1'(t)
  &=\gamma_1 x_1(t-\tau_1)
 \exp\left(
    -\mu_1\tau_1-\displaystyle\int_{t-\tau_1}^{t}\kappa_1x_1(s)+\alpha_2 x_2(s)ds
  \right)  \\[.4em]
  &\qquad
  -\mu_1 x_1(t) -\kappa_1 x_1^2(t)-\alpha_2 x_1(t)x_2(t),
  \\
  x_2'(t)
  &=\gamma_2 x_2(t-\tau_2)
 \exp\left(
     -\mu_2\tau_2-\displaystyle\int_{t-\tau_2}^{t}\kappa_2x_2(s)+\alpha_1 x_1(s)ds
	\right)  \\[.4em]
  &\qquad
   -\mu_2 x_2(t) -\kappa_2 x_2^2(t)-\alpha_1 x_1(t)x_2(t).
\end{aligned}\end{equation}

Following the proof of Arino et.~al~\cite[Proposition 3.1]{Arino2006},
together with the observation that solutions of \eqref{deq_competition}
satisfy
\begin{equation}\label{eq:well-posed_2}
  x_i'(t) \le \gamma_ie^{-\mu_i\tau_i}x_i(t-\tau_i)-\mu_ix_i(t)-\kappa_ix_i^2(t),\quad i=1,2,
\end{equation}
it can be shown that
for any solution of \eqref{eq:well-posed_2} with 
initial data $(\phi_1(t),\phi_2(t))$
where $\phi_1,\phi_2\in C([0,\tau], \mathbb R_+\setminus\{0\})$,
a constant $M$ exists such that \[
  M> \max_{t\in [0,\tau]}\phi_i(t)
  \quad\text{and}\quad
  \gamma_ie^{-\mu_i\tau_i}-\mu_i - \kappa_i M <0
  \quad\text{for $i=1,2$},
\]
and both components $x_1(t)$ and $x_2(t)$ of the solution
remain positive and bounded above by $M$ for all $t>0$.
By the standard theory of delay differential equations,
it follows that equation~\eqref{deq_competition} is well-posed.%

We  discuss
the existence and stability of equilibria
in Sections~\ref{sec_competition_equilibria}--\ref{sec_competition_global}.
Adaptive dynamics for this model is discussed
in Section~\ref{sec_competition_adaptive}.

%%%%%%%%%%%%%%%%%%%%%%%%%%%%%%
\subsection{Equilibria}
\label{sec_competition_equilibria}
%%%%%%%%%%%%%%%%%%%%%%%%%%%%%%

For each species, the survival threshold is defined by
\begin{equation}\label{R0i}
  \R^{(i)}
  =\frac{\gamma_i e^{-\mu_i \tau_i}}{\mu_i},
  \quad\mbox{for}\;\;i=1,2.
\end{equation}
There are four possible equilibria:
the extinction equilibrium $E_0=(0,0)$;
the semi-trivial equilibrium $E_1=(x^*_1,0)$ exists if $\R^{(1)}>1$;
$E_2=(0,x^*_2)$ exists if $\R^{(2)}>1$;
and the coexistence equilibrium $E_c=(x_1^c,x_2^c)$ exists under certain conditions that we will discuss later.

For the semi-trivial equilibrium $E_i$, $i=1,2$, $x^*_i$ satisfies
\begin{equation}\label{eq:x*}
\gamma_i e^{-\tau_i(\mu_i+\kappa_ix^*_i)} =\mu_i+ \kappa_ix^*_i.
\end{equation}
The components of $E_c$ satisfy
\begin{equation}\begin{aligned}\label{eq:Ec}
  &\gamma_1 e^{-\tau_1(\mu_1+\kappa_1x_1^c+\alpha_2 x_2^c)} 
  =\mu_1+ \kappa_1x_1^c+\alpha_2 x_2^c,\\
  &\gamma_2 e^{-\tau_2(\mu_2+\kappa_2x_2^c+\alpha_1 x_1^c)} 
  =\mu_2+\kappa_2x_2^c+\alpha_1 x_1^c.
\end{aligned}\end{equation}
The first equation in \eqref{eq:Ec}
implies that the quantity $\kappa_1x_1^c+\alpha_2 x_2^c$
satisfies the equation for $\kappa_1x_1^*$ in \eqref{eq:x*}.
Similarly, the second equation in \eqref{eq:Ec}
implies that the quantity $\kappa_2x_2^c+\alpha_1 x_1^c$
satisfies the equation for $\kappa_2x_2^*$ in \eqref{eq:x*}.
Hence,
\begin{equation}\label{eq:Ec1}\begin{aligned}
  &\kappa_1x_1^c+\alpha_2 x_2^c=\kappa_1 x^*_1,\\
  &\alpha_1 x_1^c+\kappa_2 x_2^c= \kappa_2 x^*_2.
\end{aligned}\end{equation}
When $\kappa_1\kappa_2-\alpha_1\alpha_2=0$,
system \eqref{eq:Ec1} has a line of solutions in the $x_1$-$x_2$ plane.
We  ignore this marginal case.
In the case that $\kappa_1\kappa_2-\alpha_1\alpha_2\ne 0$,
system \eqref{eq:Ec1} has a unique nonzero solution
\begin{equation}\notag
  x_1^c
  =\kappa_2\frac{\kappa_1x^*_1-\alpha_2 x^*_2}
  {\kappa_1\kappa_2-\alpha_1\alpha_2},\quad
  x_2^c
  =\kappa_1\frac{\kappa_2x^*_2-\alpha_1 x^*_1}
  {\kappa_1\kappa_2-\alpha_1\alpha_2}.
\end{equation}
Consequently, the coexistence equilibrium $E_c$
exists when $\R^{(1)}>1$, $\R^{(2)}>1$
and either the {\em weak interspecific competition condition}
\begin{equation}\label{cond_Hs}\tag{$\mathrm{H_S}$}
  \kappa_1\kappa_2>\alpha_1\alpha_2,\;
  \kappa_1x^*_1>\alpha_2 x^*_2,
  \;\;\;\text{and}\;\;\;
  \kappa_2x^*_2>\alpha_1 x^*_1
\end{equation}
or the {\em strong interspecific competition condition}
\begin{equation}\label{cond_Hu}\tag{$\mathrm{H_U}$}
  \kappa_1\kappa_2<\alpha_1\alpha_2,\;\;
  \kappa_1x^*_1<\alpha_2 x^*_2,
  \;\;\;\text{and}\;\;\;
  \kappa_2x^*_2<\alpha_1 x^*_1
\end{equation}
holds.
In Section~\ref{sec_competition_local}
we will see that $E_c$
is locally asymptotically stable if \eqref{cond_Hs} holds
and is unstable if \eqref{cond_Hu} holds.

%%%%%%%%%%%%%%%%%%%%%%%%%%%%%%
\subsection{Local Stability}
\label{sec_competition_local}
%%%%%%%%%%%%%%%%%%%%%%%%%%%%%%

The criteria of local stability of
all possible nonnegative equilibria of \eqref{deq_competition}
are listed as follows.

\begin{proposition}
\label{prop_competition_local}
Consider system~\eqref{deq_competition}.
Let $E_0$, $E_1$, $E_2$ and $E_c$
be the possible equilibria
given in Section~\ref{sec_competition_equilibria}.
\begin{enumerate}[(a)]
  \item $E_0$
  is locally asymptotically stable if
  $\R^{(1)}<1$ and $\R^{(2)}<1$;
  unstable if either $\R^{(1)}>1$ or $\R^{(2)}>1$.
  \item If $E_1$ exists,
  it is locally asymptotically stable if
  $\alpha_1 x_1^*>\kappa_2x_2^*$;
  unstable if $\alpha_1 x_1^*<\kappa_2x_2^*$
  \item If $E_1$ exists,
  it is locally asymptotically stable if
  $\alpha_2 x_2^*>\kappa_1x_1^*$;
  unstable if $\alpha_2 x_2^*<\kappa_1x_1^*$
  \item If a unique $E_c$ exists,
  it is locally asymptotically stable if
  \eqref{cond_Hs} holds  and unstable if \eqref{cond_Hu} holds.
\end{enumerate}
\end{proposition}

The results of this proposition are summarized in Table~\ref{table_competition}.

\begin{table}[hbtp]
\caption{Criteria of the existence and local stability
of all possible nonnegative equilibria of system~\eqref{deq_competition}.}
\centering
\begin{tabular}{p{6em}cp{10em}}
  \hline
  Equilibrium  &  Existence  & Stability Conditions
  \\
  \hline
  \\[-1em]
  $E_0=(0,0)$  &  Always &
  $\R^{(1)}<1$ and $\R^{(2)}<1$
  \\[.2em]
  $E_1=(x_1^*,0)$  & $\R^{(1)}>1$  &  $\alpha_1 x_1^*>\kappa_2x_2^*$
  \\[.2em]
  $E_2=(0,x_2^*)$  & $\R^{(2)}>1$  &  $\alpha_2 x_2^*>\kappa_1x_1^*$
  \\[.2em]
  $E_c=(x_c,y_c)$
  &
  \parbox{13em}{\centering
    $\R^{(1)}>1,\ \R^{(2)}>1$ and
    \\
    either \eqref{cond_Hs} or \eqref{cond_Hu} holds
  }
  &
  \eqref{cond_Hs}
  \\[.8em]
  \hline
\end{tabular}
\label{table_competition}
\end{table}

\begin{proof}
Let $E=(\bar{x}_{1},\bar{x}_{2})$ be an equilibrium of \eqref{deq_competition}.
Then the corresponding Jacobian matrix at $E$ is
\begin{equation}\label{ME_x12}\begin{aligned}
   M(E)=&\begin{pmatrix}
    \gamma_1 e^{-\lambda\tau_1-\tau_1m_1}
    -\mu_1-2\kappa_1\bar{x}_{1}-\alpha_2\bar{x}_{2}
    & -\alpha_2\bar{x}_{1}\\
    -\alpha_1\bar{x}_{2}
    &\gamma_2 e^{-\lambda\tau_2-\tau_2m_2}
    -\mu_2-\alpha_1\bar{x}_{1}- 2\kappa_2\bar{x}_{2}
  \end{pmatrix}\\[.5em]
  &\quad+\begin{pmatrix}
    \gamma_1e^{-\tau_1m_1}\frac{e^{-\lambda\tau_1}-1}{\lambda}\kappa_1\bar{x}_1
    &\gamma_1e^{-\tau_1m_1}\frac{e^{-\lambda\tau_1}-1}{\lambda}\alpha_2\bar{x}_1\\[.5em]
    \gamma_2e^{-\tau_2m_2}\frac{e^{-\lambda\tau_2}-1}{\lambda}\alpha_1\bar{x}_2
    & \gamma_2e^{-\tau_2m_2}\frac{e^{-\lambda\tau_2}-1}{\lambda}\kappa_2\bar{x}_2
  \end{pmatrix}-\lambda \begin{pmatrix} 1& 0\\ 0&1  \end{pmatrix}
\end{aligned}\end{equation}
where \begin{equation}\label{def_m12}\begin{aligned}
  m_1=\mu_1+\kappa_1\bar{x}_{1}+\alpha_2\bar{x}_{2},\quad
  m_2=\mu_2+\alpha_1\bar{x}_{1}+\kappa_2\bar{x}_{2}.
\end{aligned}\end{equation}

At the equilibrium $E_0=(0,0)$,
\begin{equation}\notag\begin{aligned}
  M(E_0)
  &=\det\begin{pmatrix}
    \gamma_1 e^{-\lambda\tau_1-\tau_1\mu_1} -\mu_1-\lambda
    & 0\\
    0
    &\gamma_2 e^{-\lambda\tau_2-\tau_2\mu_2} -\mu_2-\lambda
  \end{pmatrix}.\\[.5em]
\end{aligned}\end{equation}
The eigenvalues of $M(E_0)$ are $\mu_1\R^{(1)}$ and $\mu_2\R^{(2)}$.
Hence $E_0$ is locally asymptotically stable
if and only if $\R^{(1)}<1$ and $\R^{(2)}<1$,
and is unstable if $\R^{(1)}>1$ or $\R^{(2)}> 1$,

When $\R^{(1)}>1$, the semi-trivial equilibrium $E_1=(x_1^*,0)$ exists.
From \eqref{eq:x*}, the number $m_1$ defined in \eqref{def_m12} satisfies
\[
  m_1= \gamma_1e^{-\tau_1m_1}.
\]
Hence equation~\eqref{ME_x12} gives
\begin{equation}\notag\begin{aligned}
  M(E_1)
  &=\begin{pmatrix}
   m_1 e^{-\lambda\tau_1}
    -\mu_1-2\kappa_1x_{1}^*
    &\;\; -\alpha_2x_{1}^*\\
    0
    &\;\;\gamma_2 e^{-\lambda\tau_2-\tau_2m_2}
    -\mu_2-\alpha_1x_{1}^*
  \end{pmatrix}\\[.5em]
  &\qquad
  +\begin{pmatrix}
    m_1\frac{e^{-\lambda\tau_1}-1}{\lambda}\kappa_1x_1^*
    &\;\;\gamma_1e^{-\tau_1m_1}\frac{e^{-\lambda\tau_1}-1}{\lambda}\alpha_2x_1^*\\[.5em]
    0
    &0
  \end{pmatrix}
-\lambda \begin{pmatrix} 1& 0\\ 0&1  \end{pmatrix}.
\end{aligned}\end{equation}
Since $M(E_1)$ is a upper triangular matrix, the roots for $\det(M(E_1))=0$
are \[
  \lambda_1=m_1\left[e^{-\lambda\tau_1}+\frac{e^{-\lambda\tau_1}-1}{\lambda}\kappa_1x_1^* \right]-\mu_1-2\kappa_1x_{1}^*
\] and \[
  \lambda_2
  =\gamma_2 e^{-\lambda\tau_2-\tau_2m_2}
  -\mu_2-\alpha_1x_{1}^*.
\]
The real part of $\lambda_1$ is negative
by the proof of Lemma \ref{lem_local_single}(b).
Hence $E_1$ is locally asymptotically stable
if and only if the real part of $\lambda_2$ is negative, that is \[
  \gamma_2 e^{-\tau_2(\mu_2+\alpha_1x_1^*)}
  -(\mu_2+ \alpha_1x_{1}^*)<0,
\] or \[
  \alpha_1 x_1^*>\kappa_2x_2^*.
\]
Similarly, the equilibrium $E_2$ is locally asymptotically stable
if and only if \[
  \alpha_2 x_2^*>\kappa_1 x_1^*.
\]

When $\R^{(1)},\ \R^{(2)}>1$,
the coexistence equilibrium, $E_c=(x_1^c,x_2^c)$.
From \eqref{eq:Ec},
the numbers $m_1$ and $m_2$ defined in \eqref{def_m12} satisfy
\begin{equation}\notag\begin{aligned}
  m_1=\gamma_1e^{-\tau_1m_1},\quad
  m_2=\gamma_2e^{-\tau_2m_2}.
\end{aligned}\end{equation}
Hence equation~\eqref{ME_x12} gives
\begin{equation}\notag\begin{aligned}
  M(E^c)
  &=\begin{pmatrix}
    -m_1-\kappa_1x_{1}^*
    & -\alpha_2x_{1}^*\\
    -\alpha_1x_{2}^*
    &-m_2-\kappa_2x_{2}^*
  \end{pmatrix}
  +\begin{pmatrix}
    e^{-\lambda\tau_1}m_1&0\\
    0& e^{-\lambda\tau_2}m_2
  \end{pmatrix}
  \\[.5em]
  &\qquad
  +\begin{pmatrix}
    \frac{e^{-\lambda\tau_1}-1}{\lambda}m_1\kappa_1x_1^*
    & \;\;\frac{e^{-\lambda\tau_1}-1}{\lambda}m_1\alpha_2x_1^*\\[.5em]
    \frac{e^{-\lambda\tau_2}-1}{\lambda}m_2\alpha_1x_2^*
    & \;\;\frac{e^{-\lambda\tau_2}-1}{\lambda}m_2\kappa_2x_2^*
  \end{pmatrix}-\lambda \begin{pmatrix} 1& 0\\ 0& 1  \end{pmatrix}.
\end{aligned}\end{equation}
By a straightforward computation,
the characteristic equation $\det(M(E^c))=0$ can be written as
\begin{equation}\notag\begin{aligned}
  \Big( m_1 \frac{e^{-\lambda\tau_1}-1}{\lambda}-1 \Big)\Big( m_2 \frac{e^{-\lambda\tau_2}-1}{\lambda}-1 \Big)\Big[(\lambda+\kappa_1 x_1^c)(\lambda+\kappa_2 x_2^c)-\alpha_1\alpha_2 x_1^cx_2^c\Big]=0
\end{aligned}\end{equation}
Note that the terms $\left(m_i \frac{e^{-\lambda\tau_i}-1}{\lambda}-1\right)$, $i=1,2$,
as shown in the proof of Lemma~\ref{lem_local_single},
have no root with non-negative real parts,
so all roots of $\det(M(E^c))=0$
with non-negative real parts are roots of the quadratic equation
\[
  \lambda^2
  +(\kappa_1x_1^c+\kappa_2 x_2^c)\lambda
  +x_1^c x_2^c(\kappa_1\kappa_2
  -\alpha_1\alpha_2)=0.
\]
Since $\kappa_1x_1^c+\kappa_2 x_2^c>0$,
this has no roots with positive real parts if
$\kappa_1\kappa_2>\alpha_1\alpha_2$,
and has exactly one root with a positive real part if
$\kappa_1\kappa_2<\alpha_1\alpha_2$.
Therefore, we conclude  that when $E_c$ exits, it 
is locally asymptotically stable when
there is weak competition between the competitors,
and it is unstable with one dimensional stable manifold
under the condition for strong competition.%
\qed
\end{proof}

%%%%%%%%%%%%%%%%%%%%%%%%%%%%%%
\subsection{Some Results on Global Dynamics}
\label{sec_competition_global}
%%%%%%%%%%%%%%%%%%%%%%%%%%%%%%

In this section we study the global dynamics of system \eqref{deq_competition}.
First we show that a necessary condition for the $i$th species
to survive is $\R^{(i)}>1$,
where $i=1,2$ and $\R^{(i)}$ is defined by \eqref{R0i}.

\begin{theorem}
\label{thm_R0i_small}
If $\R^{(i)}\leq1$, where $i\in \{1,2\}$, then $\lim_{t\to\infty}x_i(t)=0$.
\end{theorem}
\begin{proof}
We consider only the case $\R^{(1)}\leq1$
since the case $\R^{(2)}\leq1$ can be treated similarly.
We proceed by proof by contradiction. Suppose that $x_1(t)$ does not converge to $0$.
Denote $\bar{x}_1=\limsup_{t\to\infty}x_1(t)>0$.
Then either (i) $x_1(t)$ converges to $\bar{x}_1$
or (ii)
there exists an increasing sequence of times $\{t_n\}$ such that
$x_1'(t_n)=0$ and $\lim_{n\to\infty}x_1(t_n)=\bar{x}_1$.

	In case (i), for any $\varepsilon>0$
we have $\bar{x}_1-\varepsilon< x_1(t)< \bar{x}_1+\varepsilon$
for all sufficiently large $t$.
From \eqref{deq_competition} it follows that \begin{equation}\notag
  x_1'(t) <
  \gamma_1 (\bar{x}_1+\varepsilon) e^{-\mu_1\tau_1-(\bar{x}_1-\varepsilon)\tau_1}
  - \mu_1(\bar{x}_1-\varepsilon).
\end{equation}
Note that as $\varepsilon\to 0$
the right-hand side of the inequality converges to
\[\begin{aligned}
  \gamma_1\bar{x}_1e^{-\mu_1\tau_1-\bar{x}_1\tau_1}-\mu_1\bar{x}_1
  &< \gamma_1\bar{x}_1e^{-\mu_1\tau_1}-\mu_1\bar{x}_1
  \\
  &=\mu_1\bar{x}_1(\mathcal{R}_0^{(1)}-1)
  \le 0.
\end{aligned}\]
This implies that  $\limsup x_1'(t)<0$, contradicting 
the assumption that $x_1(t)$ converges to $\bar{x}_1$.

	In case (ii),
\[\begin{aligned}
    0&=x_1'(t_n)\\
    &=\gamma_1 x_1(t_n-\tau_1)\exp\left(
        -\mu_1\tau_1-\kappa_1\int_{t_n-\tau_1}^{t_n}x_1(s)+\alpha_1 x_2(s)ds
    \right)
    \\
    &\qquad
     -\mu_1 x_1(t_n) - \kappa_1x_1^2(t_n)-\alpha_2 x_1(t_n)x_2(t_n)
     \\
    &<\gamma_1 x_1(t_n-\tau_1)\exp\left(-\mu_1\tau_1\right) 
    -\mu_1 x_1(t_n) - \kappa_1x_1^2(t_n)-\alpha_1 x_1(t_n)x_2(t_n).
\end{aligned}
\]
As $n\to\infty$,
\[\begin{aligned}
    0
    &\leq \gamma_1 e^{-\mu_1\tau_1} \limsup_{n\to\infty }x_1(t_n-\tau_1)
    -\mu_1 \bar{x}_1-\kappa_1\bar{x}_1^2
    -\alpha_1\bar{x}_1  \liminf_{n\to\infty }x_2(t_n)
    \\
    &\leq \gamma_1 e^{-\mu_1\tau_1} \bar{x}_1
    -\mu_1 \bar{x}_1- \kappa_1\bar{x}_1^2
    -\alpha_1\bar{x}_1  \liminf_{n\to\infty }x_2(t_n)\\
    &=\bar{x}_1\left[ \gamma_1 e^{-\mu_1\tau_1}
    -\mu_1 \right]- \kappa_1\bar{x}_1^2
    -\alpha_1\bar{x}_1  \liminf_{n\to\infty }x_2(t_n) <0,
\end{aligned}\]
 a contradiction. Thus $x_1(t)$ converges to 0.
	\qed
\end{proof}

From this theorem, we conclude that if $\R^{(i)}\leq 1$, then $x_i$ converges to 0, and  system \eqref{deq_competition} reduces to
the single species model \eqref{deq_single} studied in Section~\ref{sec_single}. 
Therefore,
when at least one of $\R^{(1)}$ and $\R^{(2)}$ is less than 1,
we have the following results.
\begin{corollary}
Consider system \eqref{deq_competition}.
Then the following assertions hold
for any positive solution $(x_1(t),x_2(t))$.
\begin{enumerate}[(a)]
  \item
  If $\R^{(1)}\leq1$ and $\R^{(2)}\leq1$,
  then $(x_1(t),x_2(t))$ converges to $E_0$.
  \item
  If $\R^{(1)}>1$ and $\R^{(2)}\leq1$,
  then $(x_1(t),x_2(t))$ converges to $E_1$.
  \item
	  If $\R^{(1)}\leq1$ and $\R^{(2)}<1$,
  then $(x_1(t),x_2(t))$ converges to $E_2$.
\end{enumerate}
\end{corollary}

From the local stability analysis in Section~\ref{sec_competition_local}
and the simulations shown later in this section,
we propose the following  conjectures concerning the global dynamics of system~\eqref{deq_competition}.

\begin{figure}[hbtp]\centering
{\includegraphics[trim = 1cm .2cm .5cm 0cm, clip,width=\textwidth]{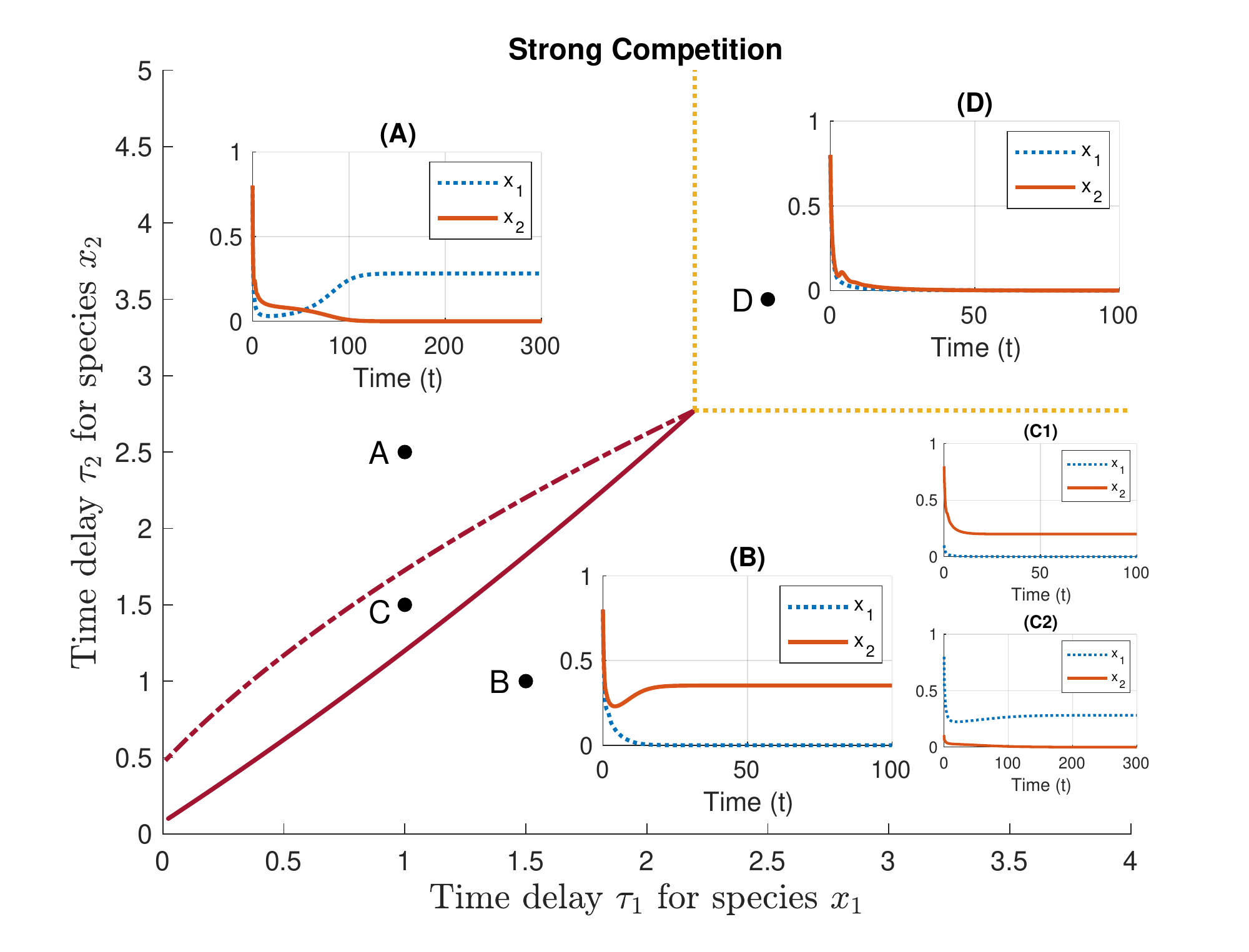}}
\caption{%
Bifurcation diagrams for system~\eqref{deq_competition}
with  strong interspecies competition, i.e., $\alpha_1\alpha_2>\kappa_1\kappa_2$.
In region $D$,
the conditions  $\R^{(1)}<1$ and $\R^{(2)}<1$ hold,
and all solutions converge to $E_0$.
Region $C$
is bounded above by the curve $\alpha_1x_1^*=\kappa_2x_2^*$
and bounded below by the curve $\alpha_2x_2^*=\kappa_1x_1^*$.
Condition~\eqref{cond_Hu} is satisfied in this region,
and solutions converge to either $E_1$ or $E_2$.
Regions A and $B$ corresponds to cases (a) and (b) in Conjecture~\ref{conj_competition}.
For $(\tau_1,\tau_2)$ in region $A$,
only $x_1$ survives;
for $(\tau_1,\tau_2)$ in region $B$,
only $x_2$ survives.%
}
\label{fig_strong}
\end{figure}

\begin{figure}[hbtp]\centering
{\includegraphics[trim = 1cm .2cm .5cm 0cm, clip,width=\textwidth]{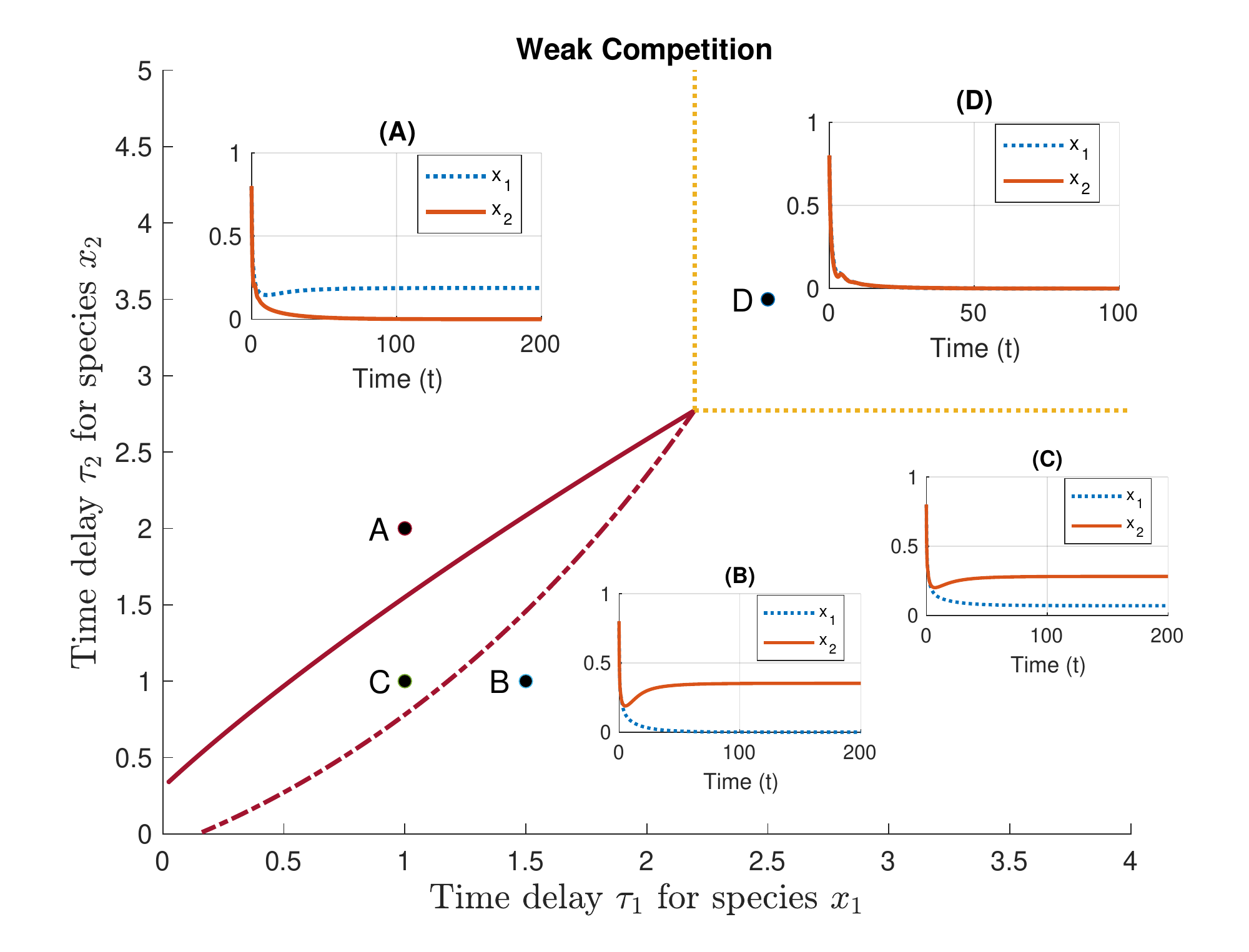}}
\caption{%
Bifurcation diagrams for system~\eqref{deq_competition}
with  weak interspecies competition, i.e.,  $\alpha_1\alpha_2<\kappa_1\kappa_2$.
Regions $A$, $B$ and $D$
are similar to the corresponding regions in Fig.~\ref{fig_strong}.
Region $C$
is bounded above by the curve $\alpha_2x_2^*=\kappa_1x_1^*$
and bounded below by the curve $\alpha_1x_1^*=\kappa_2x_2^*$.
Condition~\eqref{cond_Hs} is satisfied in this region,
and all solutions converge to the coexistence equilibrium $E_c$.%
}
\label{fig_weak}
\end{figure}

\begin{conjecture}
\label{conj_competition}
Consider system~\eqref{deq_competition}.
Assume $\R^{(1)}>1$ and $\R^{(2)}>1$.
Then, for any solution $(x_1(t),x_2(t))$ with
positive initial data, the following assertions hold:
	\begin{enumerate}[{\it (a)}]
  \item
  If $\alpha_1 x_1^*>\kappa_2  x_2^*$ and $\kappa_1 x_1^*> \alpha_2 x_2^*$,
  then $(x_1(t),x_2(t))$ converges to $E_1$.
  \item
  If $\alpha_1 x_1^*<\kappa_2  x_2^*$ and $\kappa_1x_1^*< \alpha_2 x_2^*$,
  then $(x_1(t),x_2(t))$ converges to $E_2$.
  \item
  If \eqref{cond_Hs} holds,
  then $(x_1(t),x_2(t))$ converges to $E_c$.
  \item
  If \eqref{cond_Hu} holds,
  then for any initial data
  that is not on the one-dimensional stable manifold of $E_c$,
  the solution $(x_1(t),x_2(t))$ converges to either $E_1$ or  $E_2$.
\end{enumerate}
\end{conjecture}

When time delays are small enough ($\tau<\tau_H^i$),
then our model of  competition between two species has similar outcomes
to the corresponding classical Lotka-Volterra
competition  ODE model. 
However, when delays are large, i.e.,  $\tau>\tau_H^i$,
in our delay model,  species $x_i$  dies out.
That does not happen in the corresponding ODE model.
In Wolkowicz and Xia~\cite{Wolkowicz1997},
the same conclusion was obtained
for their two species competition model with delay in a chemostat.

Bifurcation diagrams and numerical simulations
that illustrate Conjecture~\ref{conj_competition}
are shown in Figs.~\ref{fig_strong} and \ref{fig_weak}.
In both figures,
we take $(\gamma_1,\mu_1,\gamma_2,\mu_2)=(2,0.5,1.5,0.5)$
and use the delays $\tau_1$ and $\tau_2$ as bifurcation parameters.
In Fig.~\ref{fig_strong},
we take $(\alpha_1,\alpha_2,\kappa_1,\kappa_2)=(1,1.5,0.8,1)$,
which satisfies the strong interspecies competition condition $\alpha_1\alpha_2>\kappa_1\kappa_2$.
In region $D$, where both $\tau_1$ and $\tau_2$ are greater than certain critical values,
the conditions that $\R^{(1)}<1$ and $\R^{(2)}<1$ hold,
so solutions of system~\eqref{deq_competition}
converge to $E_0$ according to Theorem~\ref{thm_R0i_small}.
Region $C$,
bounded above by the curve $\alpha_1x_1^*=\kappa_2x_2^*$
and bounded below by the curve $\alpha_2x_2^*=\kappa_1x_1^*$
is where condition~\eqref{cond_Hu} is satisfied.
We choose $(\tau_1,\tau_2)=(1,1.5)$ in region $C$.
With the constant initial data $(x_1,x_2)=(0.8,0.1)$,
only $x_1$ survives;
with the constant initial data $(x_1,x_2)=(0.1,0.8)$
only $x_2$ survives.
Regions $A$ and $B$ correspond to cases (a) and (b) in Conjecture~\ref{conj_competition}.
We chose $(\tau_1,\tau_2)=(1,2)$ in region $A$
and verified that only $x_1$ survives.
We chose $(\tau_1,\tau_2)=(1.5,1)$ in region $B$
and verified that only $x_2$ survives.

In Fig.~\ref{fig_weak},
we take $(\alpha_1,\alpha_2,\kappa_1,\kappa_2)=(1,0.5,1.2,1)$,
which satisfies the weak interspecies competition condition $\alpha_1\alpha_2<\kappa_1\kappa_2$.
Regions $A$, $B$, and $D$
are similar to the corresponding regions in Fig.~\ref{fig_strong}.
Region $C$,
bounded above by the curve $\alpha_2x_2^*=\kappa_1x_1^*$
and bounded below by the curve $\alpha_1x_1^*=\kappa_2x_2^*$
is where condition~\eqref{cond_Hs} is satisfied.
We choose $(\tau_1,\tau_2)=(1,1)$ in region $C$.
With the constant initial data $(x_1,x_2)=(0.8,0.8)$,
we verified that both $x_1$ and $x_2$ survive.

%%%%%%%%%%%%%%%%%%%%%%%%%%%%%%
\subsection{Adaptive Dynamics}
\label{sec_competition_adaptive}
%%%%%%%%%%%%%%%%%%%%%%%%%%%%%%

Adaptive dynamics is a set of techniques that can be used to predict how traits evolve.
In this section, we use adaptive dynamics 
to consider  
how the length of the delay, $\tau$, in our model is predicted to evolve.

Assume that  the resident  species is denoted by $x_1$ and the mutant species
is denoted by $x_2$.
Assume also that   before the arrival of any members of the
mutant species, the resident population size has
converged to   
its delay reduced carrying capacity $x_1^*$ given by \eqref{eq:x*}.
Then the equation for $x_2'(t)$  at the time of the arrival of the first
mutants given by \eqref{deq_competition}  can be assumed to satisfy:
\[
\begin{aligned}
  x_2'
  &=\gamma_2 \exp\left(
    -\mu_2\tau_2-\kappa_2\int_{t-\tau_2}^{t} x_2(s)ds-\alpha_1 x_1^*\tau_2
  \right)x_2(t-\tau_2)
  \\
  &\qquad
  -\mu_2 x_2(t) - \kappa_2x_2^2(t)-\alpha_1 x_1^*x_2(t).
\end{aligned}
\]
Thus, $x_2$ can invade if
\[
  0
  <\gamma_2 e^{-\tau_2(\mu_2+\alpha_1 x_1^*)} -\mu_2 -\alpha_1 x_1^*,
\]
or, equivalently, using \eqref{eq:x*}, 
\[
  \kappa_2 x_2^*>\alpha_1x_1^*.
\]
Hence, the invasion exponent for mutant $x_2$
(see e.g.~Diekmann~\cite{Diekmann2004})
is \begin{equation}\label{r_x1x2}
  r_{x_1^*}(x_2^*)=\kappa_2x_2^*-\alpha_1x_1^*.
\end{equation}
Notice that $x_2$ can invade if $r_{x_1^*}(x_2^*)<0$.

We next assume that  the resident and mutant populations are identical except for their delay,
and thus, $\gamma_1=\gamma_2:=\gamma$, $\mu_1=\mu_2:=\mu$
and $\kappa_1=\kappa_2=\alpha_1=\alpha_2:=\kappa$.
From \eqref{r_x1x2}, it follows that $x_2$ can invade if $x_2^*>x_1^*$.
Without loss of generality, we scale $\kappa\mapsto 1$ in this section.
We consider the delay reduced carrying capacity as a function of the time delay.
Then the positive equilibrium of \eqref{deq_single}, denoted by $x^*(\tau)$,
is given by
\begin{equation}\notag
\gamma e^{-\tau (\mu+x^*(\tau))}=\mu+x^*(\tau).
\end{equation}
Differentiating this equation implicitly yields
\[\frac{d x^*(\tau)}{d\tau}=\frac{-(\mu+x^*(\tau))\gamma e^{-\tau (\mu+x^*(\tau))}}{1+\gamma\tau e^{-\tau (\mu+x^*(\tau))}} <0.
\]
Hence, the invasion exponent $r_{x_1^*}(x_2^*)=x^*(\tau_2) - x^*(\tau_1)$
is negative if and only if $\tau_2<\tau_1$.
Thus, the following result follows.

\begin{result}
If resident and mutant  species have identical parameters
except for the delay,
then if the mutant that takes a strategy with shorter delay than the resident
it would have a  larger delay reduced carrying capacity
and would be able to  invade successfully.
Thus, the evolutionary trend is to make the time delay as short as possible.
\end{result}

Next we discuss the situation
where there is a trade-off between the growth rate and time delay.
Motivated by the assumption that the species with shorter delay produces
less newborns than the species with longer delay,
we assume that the growth rate is an increasing function of time delay $\tau$.
More specifically, we let $\gamma(\tau)=\gamma_0(1+c\tau)$,
where $\gamma_0$ and $c$ are positive constants.
For simplicity,
we assume that $\mu_1=\mu_2:=\mu$
and that $\kappa_i=\alpha_i=1$.
Then equation~\eqref{def_xbar} becomes \begin{equation}\label{eq_xbarc}
  \gamma_0(1+c\,\tau)e^{-\tau(\mu+\kappa x)}-(\mu+x)=0,
\end{equation}
and the threshold value $\tau_H$
for the existence of $x^*$ is determined by
\begin{equation}\label{eq_tauHc}
  \gamma_0(1+c\,\tau_H)e^{-\mu\tau_H}=\mu.
\end{equation}
We assume that $\gamma_0>\mu$
to ensure that the threshold value $\tau_H$
defined by \eqref{eq_tauHc} is positive.
Differentiating  \eqref{eq_xbarc}
implicitly yields
\begin{equation}\label{eq_dxbarc}
  \frac{d x^*(\tau)}{d\tau}
  =\frac{\gamma_0\, e^{-\tau (\mu+x^*(\tau))}}
  {1+\gamma_0\,\tau e^{-\tau (\mu+x^*(\tau))}}
  \left[c-(1+c\tau)(\mu+x^*(\tau))\right].
\end{equation}
Setting $\tau=0$ in \eqref{eq_xbarc} gives $x^*(0)=\gamma_0-\mu$.
By \eqref{eq_dxbarc} it follows that
$\frac{dx^*}{d\tau}(0)$ has the same sign as $c-\gamma_0$.
Note also that $\frac{dx^*}{d\tau}(\tau_H)\le 0$
since $x^*(\tau_H)=0$ and $x^*(\tau)>0$ for $\tau<\tau_H$.
Therefore, in the case that $c>\gamma_0$,
there exists at least one $\tau^*<\tau_H$
such that $\tau^*>0$ is a critical point of $x^*(\tau)$.
In terms of Evolutionary Game Theory
(see e.g.\ Vicent and Brown~\cite{Vincent2005}),
all members of a population adopting $\tau=\tau^*$
is an evolutionarily stable strategy (ESS),
which means that
no mutant strategy could invade the population
under the influence of natural selection.

\begin{result}
If the growth rate 
and the delay are linearly positively correlated and
the growth rate with no delay is large enough,
then there is
a critical value $\tau^*>0$
such that taking $\tau$ to be $\tau^*$
is an ESS.
\end{result}

Consider for example $c=8$, $\gamma_0=3$ and $\mu=2$.
The graph of $x^*(\tau)$ is shown in Fig.~\ref{fig:adaptive}.
In this case,
the function $x^*(\tau)$ has a global maximum point $\tau^*$.
Thus, for the evolution of trait $\tau$,
the best strategy is to make the delay as close to  $\tau^*$ as possible.

\begin{figure}[hbtp]\centering
{\includegraphics[trim = 0cm .5cm 1cm 1cm, clip, width=0.58\textwidth]{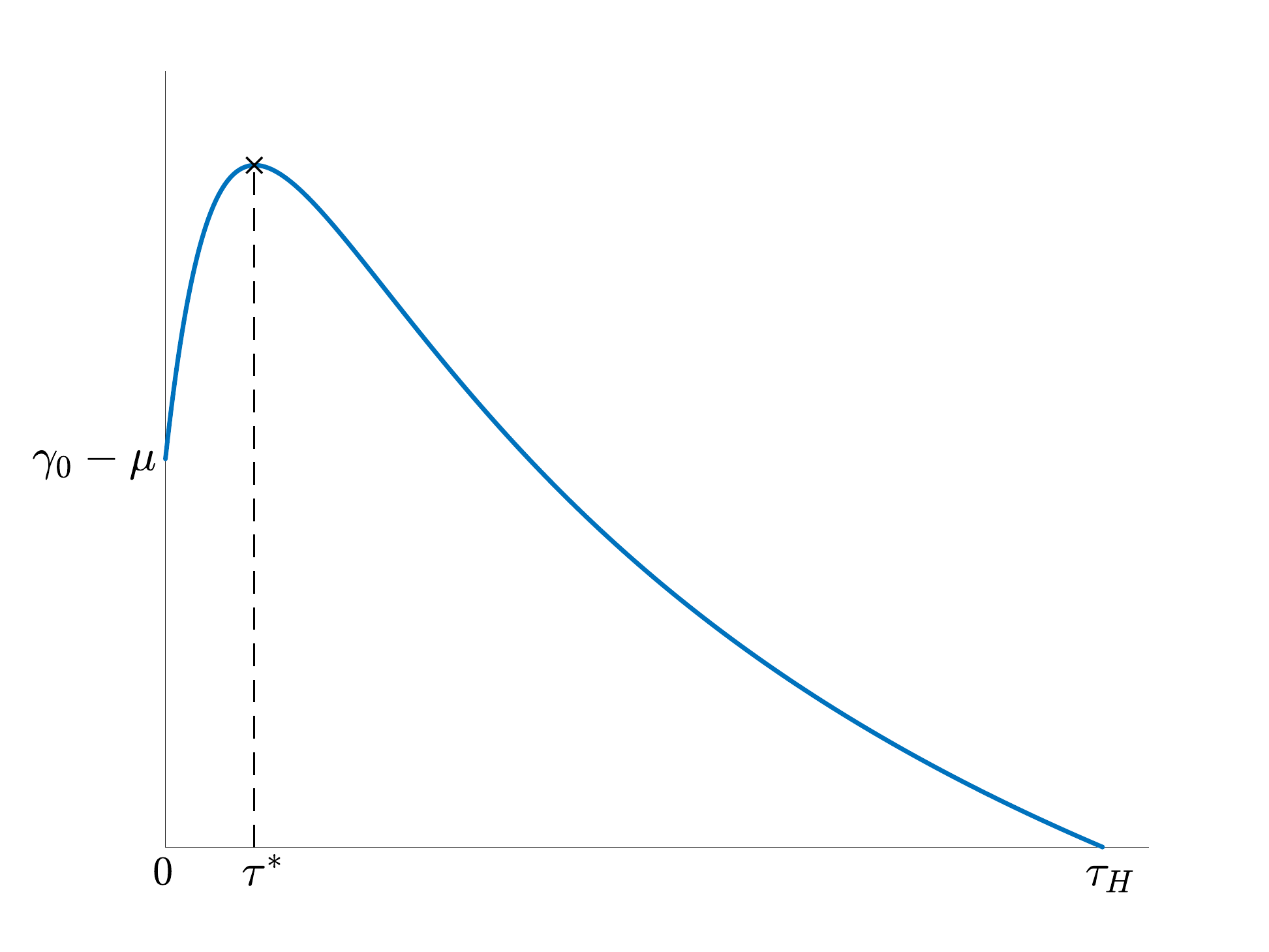}}
\caption{The graph of $x^*(\tau)$,
where $x^*(\tau)$ is 
the solution of \eqref{def_xbar}
with $\gamma(\tau)=\gamma_0(1+c\tau)$.
When $c=8$, $\gamma_0=3$ and $\mu=2$,
the threshold delay is $\tau_H\approx 1.48$,
and the graph has a unique local maximum $\tau^*\approx 0.14$
in the interval $(0,\tau_H)$.
}
\label{fig:adaptive}
\end{figure}

%%%%%%%%%%%%%%%%%%%%%%%%%%%%%%
\section{Conclusion and Discussion}
\label{sec_conclusion}
%%%%%%%%%%%%%%%%%%%%%%%%%%%%%%

Based on the assumption of
the decay-consistent delay in growth
for model \eqref{deq:A2006}
in Arino et al.~\cite{Arino2006},
we derived the novel single species delayed model \eqref{deq_single}.
The main difference 
between these two models
is whether juveniles and adults compete with each other.
In model \eqref{deq:A2006}
the mature subgroup only competes within that subgroup.
In contrast,
in our model \eqref{deq_single}
the mature subgroup competes with the whole population.
These two models fit different types of behavior.
Model~\eqref{deq:A2006}
is suitable for species
whose development includes several stages
(egg, larva, pupa, and adult)
such as holometabolous insects,
for which juveniles do not compete with
adults because they have different living environments,
and there is no intraspecific competition (crowding or direct interference)
between juveniles and adults.
On the other hand,
model \eqref{deq_single}
is suitable for most mammalian species,
for which juveniles and adults share the same environment.

All solutions  of our logistic growth DDE \eqref{deq_single}
 converge to an equilibrium
with value depending on the delay.
If the delay is too long,
this model predicts that the population dies out.
A threshold giving the interface between extinction and survival
is determined in terms of parameters in the model.
Our model and model~\eqref{deq:A2006}
have similar dynamics.
While they have different positive equilibria,
they have the same survival threshold $\R{}$,
given by \eqref{def_R0}.
The population approaches
the delay reduced carrying capacity if $\R{}>1$;
otherwise it goes to extinction.
Thus neither model has sustained oscillations
as in Hutchinson's equation \eqref{deq_Hutchinson},
which exhibits stable periodic solutions when $\tau$ is large enough
(see e.g.\ Hale and Verduyn~\cite{Hale1993}).
Using DDEBIF-TOOL~\cite{Engelborghs2002,Sieber2014},
we show a comparison of the dynamics of
Hutchinson's equation
and models \eqref{deq:A2006} and \eqref{deq_single}
in Fig.~\ref{fig:ddebif}.

\begin{figure}[hbtp]\centering
  \includegraphics[width=0.48\textwidth]{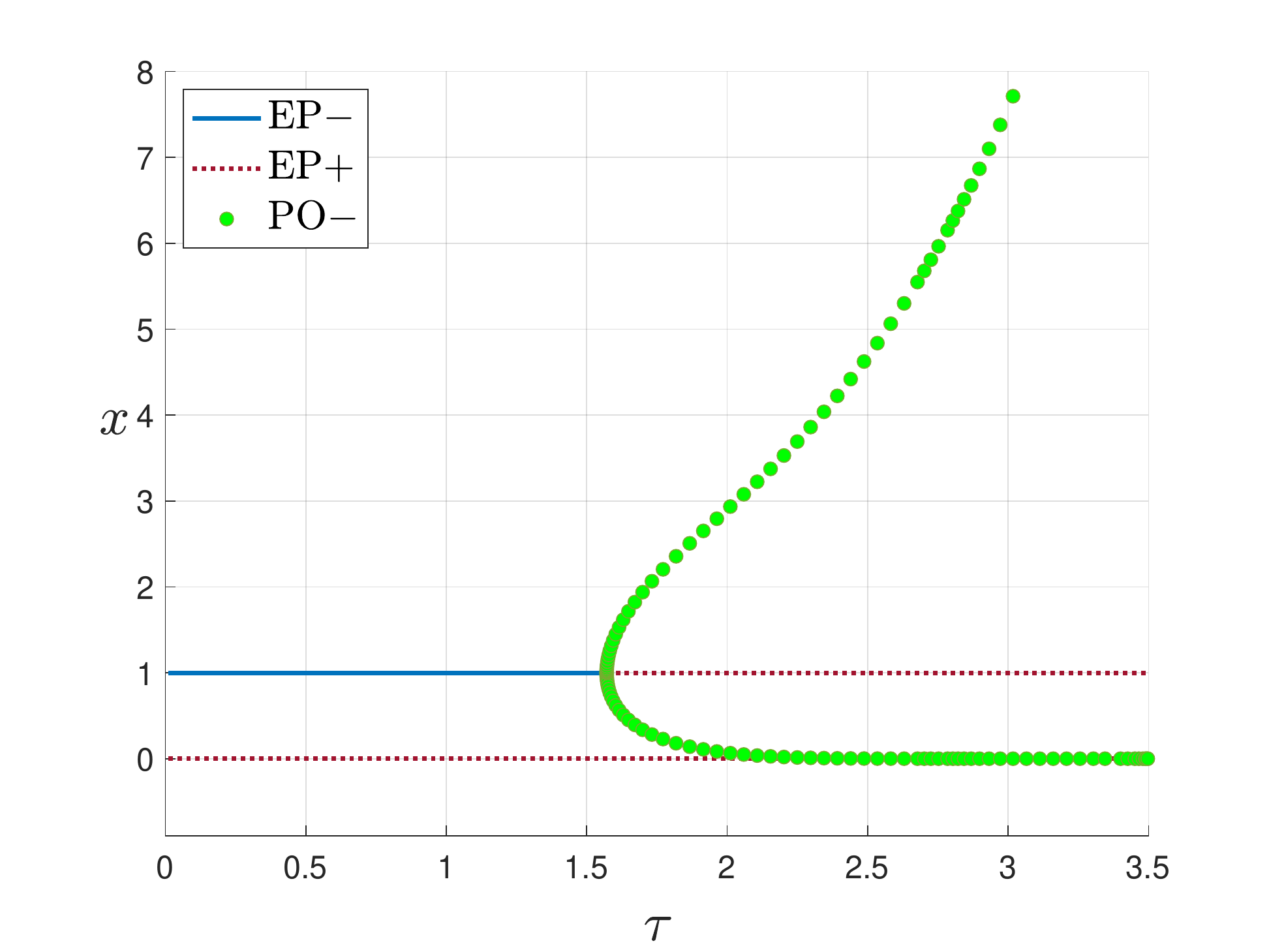}
  \includegraphics[width=0.48\textwidth]{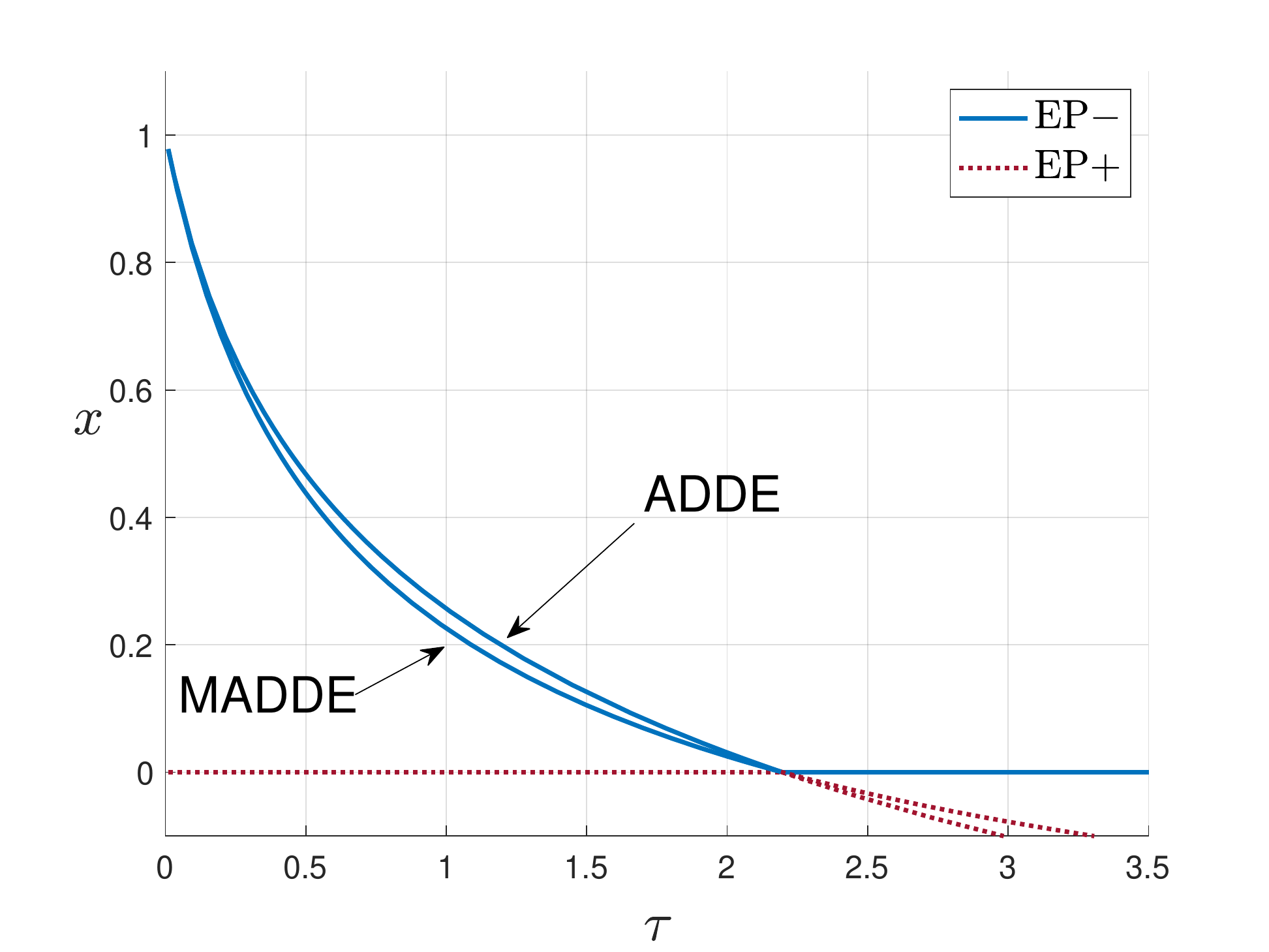}
  \\
  (a)\hspace{17em}(b)
\caption{Bifurcation diagrams for
(a) Hutchinson's equation with $r=K=1$, and
(b) the alternative logistic DDE \eqref{deq:A2006}, denoted by ADDE,
and the mixed the alternative logistic DDE \eqref{deq_single}, denoted by MADDE,
with $\gamma=1.5$, $\mu=0.5$ and $\kappa=1$.
In each figure, EP$-$
	indicates
	asymptotically stable equilibrium point,
EP$+$ indicates unstable equilibrium point,
	and PO$-$ indicates an orbitally asymptotically stable periodic orbit.
	Notice that in (a) the interior equilibrium is a constant
	function of the delay, but in (b), the interior  equilibria are decreasing
	functions of the delay $\tau$ and on the left of the transcritical
	bifurcation involving the interior equilibria and the extinction
	equilibrium,  the stable interior equilibrium for the ADDE model is
	slightly larger than the 
	stable equilibrium for the MADDE model.}
\label{fig:ddebif}
\end{figure}

Another advantage of
our derivation of logistic DDE
is that it can readily
involve more than one interacting population.
The alternative logistic DDE of Arino et al.~\cite{Arino2006}
can also be extended to a competition model
under certain conditions (see Lin et~al.~\cite{Lin2018}),
but it is unclear modify that approach in general cases.
Our competitive system \eqref{deq_competition}
was naturally generalized from our logistic DDE~\eqref{deq_single}.

For our system \eqref{deq_competition} modeling competitive interactions,
the outcome
are parallel to the two species competitive ordinary differential systems:
is either competitive exclusion holds,
two species bistable,
or the unique positive coexistence equilibrium is globally stable.
The results are parallel to that
for the competitive system \eqref{deq:Lv_competition}
in Lv et al.~\cite{Lv2017a} who 
considered a competitive system
where the maturity of a species individual
is not an instantaneous process and
 proposed the equations
\begin{equation}\label{deq:Lv_competition}
\begin{aligned}
  x_1'
  &=\gamma_1 e^{-\mu_1 \tau(x_1)}x_1(t-\tau(x_1))
  -\kappa_1x_1^2-\alpha_2 x_1 x_2,\\
  x_2'
  &=\gamma_2 e^{-\mu_2 \tau(x_2)}x_2(t-\tau(x_2))
  -\kappa_2x_2^2-\alpha_1 x_1 x_2,
\end{aligned}
\end{equation}
where $\tau(x)$ is an increasing function.

From the adaptive dynamics analysis, if the resident and mutant species are identical
except for their time delay,
then the evolutionary trend is to make the  delay as short as possible.
Thus, after long term evolution, the time delay would aproach  0, and the system
would reduce to a system of  ordinary differential equations.
However, a time delay for both reproduction and growth does exist in actual species.
Therefore, we also consider the case when the mutant has the same parameters except
for their time delay and their growth rate,
and hence there is a trade-off between the growth rate and the delay. Say the longer
delay allows more or larger newborns.
Then, there may be a positive delay $\tau^*$ that maximizes the delay reduced
carrying capacity and $0<\tau^*<\tau_H$.
In this scenario, the  delay trait in natural ecosystems would tend to move toward $\tau^*>0$ according to our model.

\section*{Acknowledgement}
 The research of Gail S. K. Wolkowicz was partially supported by a Natural Sciences and Engineering Research Council of Canada (NSERC) Discovery grant with accelerator supplement. 

%%%%%%%%%%%%%%%%%%%%%%%%%%%%%%%%%%%%%%%%%%%
%\bibliographystyle{spmpsci}
%\bibliography{ref_DDE}

\end{document}